\documentclass[11pt]{amsart}
\usepackage{amscd}
\usepackage{amsfonts}
\usepackage{color}
\usepackage{amsmath,amssymb,amsthm,latexsym}
\usepackage{amscd}

\usepackage{enumerate}
\usepackage{multicol}
\newtheorem{theorem}{Theorem}[section]
\newtheorem{proposition}[theorem]{Proposition}
\newtheorem{definition}[theorem]{Definition}
\newtheorem{corollary}[theorem]{Corollary}
\newtheorem{lemma}[theorem]{Lemma}
\usepackage{amsmath}
\usepackage{amsfonts}
\usepackage{amsmath,amsthm}
\usepackage{amscd}
\usepackage[all]{xy}
\numberwithin{equation}{section}
\theoremstyle{remark}
\newtheorem{remark}[theorem]{Remark}
\newtheorem{example}[theorem]{\bf Example}
\newcommand{\R}{\mathbb{R}}
\newcommand{\C}{\mathbb{C}}
\newcommand{\D}{\mathbb{D}}
\newcommand{\E}{\mathbb{E}}

\newcommand{\FC}{\mathcal{C}}
\newcommand{\dd}{\mathrm{d}}
\newcommand{\ee}{\mathbf{e}}

\begin{document}

\title[Willmore surfaces in spheres via loop groups]{\bf{Willmore surfaces in spheres: the DPW approach via the  conformal Gauss map}}
\author{Josef F. Dorfmeister, Peng Wang }

\date{}
\maketitle

\begin{abstract}The paper builds a DPW approach of Willmore surfaces via conformal Gauss maps. As applications, we provide descriptions of minimal surfaces in $\mathbb R^{n+2}$, isotropic surfaces in $S^4$ and homogeneous Willmore tori via the loop group  method. A new example of a Willmore two-sphere in $S^6$ without dual surfaces is also presented.
\end{abstract}

\vspace{0.5mm}  {\bf \ \ ~~Keywords:}  Willmore surface;  conformal Gauss map; normalized potential;  non-compact symmetric space; Iwasawa decomposition.     \vspace{2mm}

{\bf\   ~~ MSC(2010): \hspace{2mm} 53A30, 53C30, 53C35}


\section{Introduction}

A Willmore surface in $S^{n+2}$ is a critical surface of the Willmore functional
$\int_M(|\vec{H}|^2-K+1)dM$
with $\vec{H}$  and $K$ being the mean curvature vector and the Gauss curvature respectively. It is well-known that the Willmore functional is invariant under conformal transformations of $S^{n+2}$. Moreover, by Blaschke \cite{Blaschke}, Bryant \cite{Bryant1984}, Ejiri \cite{Ejiri1988}, and Rigoli \cite{Rigoli1987}, a conformal immersion is Willmore if and only if its conformal Gauss map is harmonic.

Later, H\'{e}lein's important observations \cite{Helein}, generalized by Xia-Shen \cite{Xia-Shen} and also developed  in a different direction by Xiang Ma \cite{Ma2006}, indicates that Willmore surfaces are also related to other  harmonic maps (into inner  symmetric spaces). Moreover, a loop group theory was built for this kind of harmonic maps in \cite{Helein,Xia-Shen}. However, this new type of harmonic maps may have singularities, which makes it unclear how to derive global properties of Willmore surfaces this way. To this end, in \cite{DoWa11} we built a way to describe the conformal harmonic maps which could globally be the conformal Gauss maps of some (branched) Willmore surfaces.

In this paper we will use such harmonic maps to study the global geometry of Willmore surfaces, via a loop group description of those harmonic maps. This includes two decomposition theorems (Iwasawa and Birkhoff) for the corresponding loop groups, the DPW procedure for our harmonic maps, as well as a (generalized) DPW construction for harmonic two-spheres (Theorem \ref{th-potential-sphere}).

 As an illustration of the method, the above results are applied  to several special  types of  Willmore surfaces, including minimal surfaces in $\mathbb R^{n+4}$, isotropic Willmore surfaces in $S^4$ and homogeneous Willmore tori in $S^{n+2}$.

Another important application of the above theory is the construction of the following example of a Willmore 2-sphere in $S^6$  in  \cite{Wang-iso}, which gives a negative answer to an open problem of Ejiri stated at the end of \cite{Ejiri1988}.
\begin{example} \label{thm-example}(\cite{Wang-iso}) Let
\[\eta=\lambda^{-1}\left(
                      \begin{array}{cc}
                        0 & \hat{B}_1 \\
                        -\hat{B}_1^tI_{1,3} & 0 \\
                      \end{array}
                    \right)\dd z,\ ~ \hbox{ with } ~\ \hat{B}_1=\frac{1}{2}\left(
                     \begin{array}{cccc}
                       2iz&  -2z & -i & 1 \\
                       -2iz&  2z & -i & 1 \\
                       -2 & -2i & -z & -iz  \\
                       2i & -2 & -iz & z  \\
                     \end{array}
                   \right).\]
Then the associated family of Willmore two-spheres $y_{\lambda}$, $\lambda\in S^1$, corresponding to $\eta$, is \begin{equation}\label{example1}
\begin{split}y_{\lambda}&=\frac{1}{ \left(1+r^2+\frac{5r^4}{4}+\frac{4r^6}{9}+\frac{r^8}{36}\right)}
\left(
                          \begin{array}{c}
                            \left(1-r^2-\frac{3r^4}{4}+\frac{4r^6}{9}-\frac{r^8}{36}\right) \\
                            -i\left(z- \bar{z})(1+\frac{r^6}{9})\right) \\
                            \left(z+\bar{z})(1+\frac{r^6}{9})\right) \\
                            -i\left((\lambda^{-1}z^2-\lambda \bar{z}^2)(1-\frac{r^4}{12})\right) \\
                            \left((\lambda^{-1}z^2+\lambda \bar{z}^2)(1-\frac{r^4}{12})\right) \\
                            -i\frac{r^2}{2}(\lambda^{-1}z-\lambda \bar{z})(1+\frac{4r^2}{3}) \\
                            \frac{r^2}{2} (\lambda^{-1}z+\lambda \bar{z})(1+\frac{4r^2}{3})  \\
                          \end{array}
                        \right) \\
  \end{split}
\end{equation}
with $r=|z|$.
  Moreover, $y_{\lambda}:S^2\rightarrow S^6$ is a Willmore immersion in $S^6$, which is non S-Willmore, full, and totally isotropic. In particular, $y_\lambda$ does not have any branch points.

Note that all the surfaces $y_{\lambda}$, $\lambda\in S^1$, are isometric to each other by rotations by matrices of $SO(7)$, since $y_{\lambda}=D_{\lambda} \cdot y_1,$ with $y_1=y_{\lambda}|_{\lambda=1}$ and
\[ D_{\lambda}=\left(
    \begin{array}{ccccccc}
      I_3 &   0 & 0 & 0 & 0 \\
      0 &   \frac{\lambda+\lambda^{-1}}{2} &  \frac{\lambda-\lambda^{-1}}{-2i} & 0 & 0 \\
      0 &   \frac{\lambda-\lambda^{-1}}{2i} &  \frac{\lambda+\lambda^{-1}}{2} & 0 & 0 \\
      0 &  0 & 0& \frac{\lambda+\lambda^{-1}}{2} &  \frac{\lambda-\lambda^{-1}}{-2i} \\
      0 &   0 & 0&  \frac{\lambda-\lambda^{-1}}{2i} &  \frac{\lambda+\lambda^{-1}}{2} \\
    \end{array}
  \right)\in SO(7).\]
\end{example}
For the meaning of $\eta$ we refer to Section 3. And we refer to \cite{Wang-iso} for detailed proofs of the above results and further discussions on $y$. This result and the characterization of the normalized potentials of all Willmore 2-spheres in $S^{n+2}$ indicates that our approach is workable for the study of global geometry of Willmore surfaces in terms of the DPW method for their conformal Gauss map.

This paper is organized as follows: in Section 2 we recall the basic theory of Willmore surfaces and harmonic maps briefly. Then in Section 3, we first recall the DPW method for harmonic maps into symmetric spaces. The Birkhoff and Iwasawa Decomposition Theorems concerning our non-compact groups, and the existence of normalized potentials for harmonic two-spheres are provided then. Section 4 is devoted to some applications of our main results for minimal surfaces in $\mathbb{R}^{n+2}$, isotropic Willmore surfaces in $S^4$ and homogeneous Willmore tori.
 We end this paper with an Appendix which contains   a proof of some properties of the loop groups used in this paper.

 {\em Remark: } Following the suggestion of some anonymous referee we have divided the paper
 entitled  ``Willmore surfaces in spheres via loop groups I: generic cases and some examples" (arXiv:1301.2756)
into three parts. The present paper is part II  and \cite{DoWa11,DoWa13} are part I and part III respectively.


\section{Willmore surfaces and harmonic maps}

In this section, we will recall the basic theory concerning Willmore surfaces and harmonic maps respectively.  For details we refer to   \cite{DoWa11}.

\subsection{Willmore surfaces and conformal Gauss maps}For completeness we first recall briefly the basic surface theory. For more details, we refer to Section 2 of \cite{DoWa11} (see also \cite{BPP}). Let $\mathbb{R}^{n+4}_1$ be the Lorentz-Minkowski space equipped with the Lorentzian metric (Here $I_{1,n+3}=diag\left(-1,1,\cdots,1\right)$)
\[\langle x,y\rangle=-x_{0}y_0+\sum_{j=1}^{n+3}x_jy_j=x^t I_{1,n+3} y, \ \forall
  x,y\in\R^{n+4}.\]
Let $\mathcal{C}^{n+3}_+= \lbrace x \in \mathbb{R}^{n+4}_{1} |\langle x,x\rangle=0 , x_0 >0 \rbrace $
be the forward light cone of $\mathbb R^){n+4}_1$ and Let $Q^{n+2}=\mathcal{C}^{n+3}_+/ \R^+$  be the  projective light cone. The three Riemannian space forms can be conformally embedded into $Q^{n+2}$ \cite{Bryant1984,Ejiri1988,BPP}.

Let $y:M\rightarrow S^{n+2}$  be a  conformal immersion. Let $z=u+iv$ be a local complex coordinate on $U\subset M$, with $|y_z|^2= \frac{1}{2}e^{2\omega}$. We call $Y=e^{-\omega}(1,y)$ a canonical lift of $y$ w.r.t. $z$ \cite{BPP, DoWa11}. There exists a global bundle decomposition $M\times \mathbb{R}^{n+4}_{1}=V\oplus V^{\perp}$,
where
\[
 V={\rm Span}\{Y,{\rm Re}Y_{z},{\rm Im}Y_{z},Y_{z\bar{z}}\},
\]
and $V^{\perp}$ denotes the orthogonal complement of $V$.  Let $V_{\mathbb{C}}$ and
$V^{\perp}_{\mathbb{C}}$ be the complexifications of $V$ and $V^{\perp}$ respectively.

Let $\{Y,Y_{z},Y_{\bar{z}},N\}$  be a  frame of
$V_{\mathbb{C}}$ such that
$
\langle N,Y_{z}\rangle=\langle N,Y_{\bar{z}}\rangle=\langle
N,N\rangle=0,\ \langle N,Y\rangle=-1.
$ Let $D$
denote the normal connection on $V_{\mathbb{C}}^{\perp}$. Then we obtain
the structure equations of $y$:
\begin{equation}\label{eq-moving}
\left\{\begin {array}{lllll}
 Y_{zz}=-\frac{s}{2}Y+\kappa,\ \\
Y _{z\bar{z}}=-\langle \kappa,\bar\kappa\rangle Y+\frac{1}{2}N,\ \\
 N_{z}=-2\langle \kappa,\bar\kappa\rangle Y_{z}-sY_{\bar{z}}+2D_{\bar{z}}\kappa,\ \\
 \psi_{z}=D_{z}\psi+2\langle \psi,D_{\bar{z}}\kappa\rangle Y-2\langle
\psi,\kappa\rangle Y_{\bar{z}}, \ \\
\end {array}\right.
\end{equation}
Here $\psi\in
\Gamma(V_{\mathbb{C}}^{\perp})$ is an arbitrary section of the conformal normal bundle.
Here $\kappa$ is  \emph{the conformal Hopf differential} of $y$
and $s$ is \emph{the Schwarzian} of $y$ \cite{BPP}.
The conformal Gauss, Codazzi and Ricci equations are
as follows:
\begin{equation}\label{eq-integ}
\left\{\begin {array}{lllll}
 \frac{1}{2}s_{\bar{z}}=3\langle
\kappa,D_z\bar\kappa\rangle +\langle D_z\kappa,\bar\kappa\rangle,\\
{\rm Im}(D_{\bar{z}}D_{\bar{z}}\kappa+\frac{\bar{s}}{2}\kappa)=0,\\
   R^{D}_{\bar{z}z}\psi=D_{\bar{z}}D_{z}\psi-D_{z}D_{\bar{z}}\psi =
  2\langle \psi,\kappa\rangle\bar{\kappa}- 2\langle
  \psi,\bar{\kappa}\rangle\kappa.
\end {array}\right.
\end{equation}

Next we define \emph{the conformal Gauss map} of $y$ (\cite{Bryant1984,BPP,Ejiri1988,DoWa11})
\begin{definition} \label{def-gauss}
Let $y:M\to S^{n+2}$ be a conformal immersion from a Riemann surface $M$. The  \emph{conformal Gauss map} of $y$ is defined by
\begin{equation}\begin{array}{ccccc}
                 Gr : & M &\rightarrow&
Gr_{1,3}(\mathbb{R}^{n+4}_{1}) &= SO^+(1,n+3)/SO^+(1,3)\times SO(n)\\
                \ & p\in M & \mapsto & V_p &\ \\
                \end{array}
\end{equation}
Here $V_p$ is  the $4-$dimensional Lorentzian subspace oriented by a basis $\{Y,N,Y_u,Y_v\}$.
\end{definition}
Note that $Gr$ can also be written as \cite{Ma2006}
$Gr=Y\wedge Y_{u}\wedge Y_{v}\wedge N=-2i\cdot Y\wedge Y_{z}\wedge
Y_{\bar{z}} \wedge N$. Let  $\{\psi_j\}$ be an orthonormal basis of $V^{\perp}$ on $U$. Then a local lift of $Gr$ is chosen as \cite{DoWa11}
\begin{equation}\label{F}
F:=\left(\frac{1}{\sqrt{2}}(Y+N),\frac{1}{\sqrt{2}}(-Y+N),e_1,e_2,\psi_1,\cdots,\psi_n\right): U\rightarrow  SO^+(1,n+3)
\end{equation}
with its Maurer-Cartan form being of the form ( See Theorem 2.2 of \cite{DoWa11} for more details.)
  \begin{equation}\label{eq-MC-1}\alpha=F^{-1}\dd F=\left(
                   \begin{array}{cc}
                     A_1 & B_1 \\
                   -B_1^tI_{1,3} & A_2 \\
                   \end{array}
                 \right)\dd z+\left(
                   \begin{array}{cc}
                     \bar{A}_1 & \bar{B}_1 \\
                    -\bar{B}_1^tI_{1,3}& \bar{A}_2 \\
                   \end{array}
                 \right)\dd\bar{z},\end{equation}
 where
 \begin{equation}\label{eq-MC-2}A_1=\left(
                             \begin{array}{cccc}
                               0 & 0 & s_1 & s_2\\
                               0 & 0 & s_3 & s_4 \\
                               s_1 & -s_3 & 0 & 0 \\
                               s_2 & -s_4 & 0 & 0 \\
                             \end{array}
                           \right),\   A_2=\left(
                             \begin{array}{cccc}
                               b_{11} & \cdots &  b_{n1} \\
                               \vdots& \vdots & \vdots \\
                               b_{1n} &\cdots & b_{nn} \\
                             \end{array}
                           \right),\ \end{equation}
\begin{equation} \label{eq-MC-3}
\left\{\begin{split}&s_1=\frac{1}{2\sqrt{2}}(1-s-2k^2),\ s_2=-\frac{i}{2\sqrt{2}}(1+s-2k^2),\\
&s_3=\frac{1}{2\sqrt{2}}(1+s+2k^2),\ s_4=-\frac{i}{2\sqrt{2}}(1-s+2k^2),\\
\end{split}\right.
\end{equation}
and
\begin{equation}\label{eq-b1} B_1=\left(
      \begin{array}{ccc}
         \sqrt{2} \beta_1 & \cdots & \sqrt{2}\beta_n \\
         -\sqrt{2} \beta_1 & \cdots & -\sqrt{2}\beta_n \\
        -k_1 & \cdots & -k_n \\
        -ik_1 & \cdots & -ik_n \\
      \end{array}
    \right).   \end{equation} Here  $\kappa=\sum_j k_j\psi_j ,\ D_{\bar{z}}\kappa=\sum_j\beta_j\psi_j,\ k^2=\sum_j|k_j|^2.$\\

 The \emph{ Willmore functional} of $y$ (See \cite{BPP,DoWa11,Ma2006}) is
defined as:
\begin{equation}\label{eq-W-energy}
W(y):=2i\int_{M}\langle \kappa,\bar{\kappa}\rangle \dd z\wedge
\dd \bar{z}.
\end{equation}
An immersed surface $y:M\rightarrow S^{n+2}$ is called a
\emph{Willmore surface}, if it is a critical point of the Willmore
functional with respect to any variation (with compact support) of $y$. We have
\begin{theorem}\label{thm-willmore} \cite{Bryant1984,BPP,Ejiri1988,Rigoli1987}
For a conformal immersion $y:M\rightarrow  S^{n+2}$, the following three conditions
are equivalent:
 \begin{enumerate}
\item $y$ is Willmore;

\item The conformal Gauss map $Gr_y$ is a conformally harmonic map into
$G_{3,1}(\mathbb{R}^{n+3}_{1})$;

\item The conformal Hopf differential $\kappa$ of $y$ satisfies the
``Willmore condition":
\begin{equation}\label{eq-willmore}
D_{\bar{z}}D_{\bar{z}}\kappa+\frac{\bar{s}}{2}\kappa=0, \hbox{ for any contractible chart of $M$.}
\end{equation}
 \end{enumerate}
\end{theorem}


\subsection{Strongly conformal harmonic maps into $SO^+(1,n+3)/SO^+(1,3)\times SO(n)$}
This subsection is to collect the basic setup of the loop group approach of harmonic maps into symmetric spaces and the strongly conformal harmonic maps related with Willmore surfaces.

\subsubsection{Harmonic maps into the symmetric space G/K}

Let $N=G/K$ be a symmetric space with involution $\sigma: G\rightarrow G$ such
that $G^{\sigma}\supset K\supset(G^{\sigma})_0$. Let $\pi:G\rightarrow G/K$ denote the projection of $G$ onto $G/K$. Let $\mathfrak{g}$ and $\mathfrak{k}$ denote the
Lie algebras of $G$ and $K$ respectively. The involution $\sigma$ induces the Cartan decomposition
$\mathfrak{g}=\mathfrak{k}\oplus\mathfrak{p},$ with $ [\mathfrak{k},\mathfrak{k}]\subset\mathfrak{k},\
 [\mathfrak{k},\mathfrak{p}]\subset\mathfrak{p}$ and $[\mathfrak{p},\mathfrak{p}]\subset\mathfrak{k}.$

Let $f:M\rightarrow G/K$ be a conformally harmonic map from a connected Riemann surface $M$.
Let $U\subset M$ be an open contractible subset.
Then there exists a frame $F: U\rightarrow G$ such that $f=\pi\circ F$ on $U$.
Let $\alpha$ denote the Maurer-Cartan form of $F$. Then $\alpha$ satisfies the Maurer-Cartan equation
and altogether we have
$F^{-1}\dd F= \alpha,$ with $\dd \alpha+\frac{1}{2}[\alpha\wedge\alpha]=0.$
Decomposing $\alpha$ with respect to $\mathfrak{g}=\mathfrak{k}\oplus\mathfrak{p}$ we obtain
\begin{equation*}\alpha=\alpha_{ \mathfrak{k}  } +\alpha_{ \mathfrak{p} },\ \hbox{ with }\
\alpha_{\mathfrak{k  }}\in \Gamma(\mathfrak{k}\otimes T^*M),\
\alpha_{ \mathfrak{p }}\in \Gamma(\mathfrak{p}\otimes T^*M).
\end{equation*}
Next we decompose $\alpha_{\mathfrak{p}}$ further into the $(1,0)-$part $\alpha_{\mathfrak{p}}'$ and the $(0,1)-$part $\alpha_{\mathfrak{p}}''$,
and set  \begin{equation}\label{eq-harmonic-lam}
\alpha_{\lambda}=\lambda^{-1}\alpha_{\mathfrak{p}}'+\alpha_{\mathfrak{k}}+\lambda\alpha_{\mathfrak{p}}'', \hspace{5mm}  \lambda\in S^1.
\end{equation}

\begin{lemma}\label{lemma-harmonic} $($\cite{DPW}$)$ The map  $f:M\rightarrow G/K$ is harmonic if and only if
\begin{equation}\label{integr}\dd
\alpha_{\lambda}+\frac{1}{2}[\alpha_{\lambda}\wedge\alpha_{\lambda}]=0,\ \ \hbox{for all}\ \lambda \in S^1.
\end{equation}
\end{lemma}

\begin{definition} Let $f:M\rightarrow G/K$ be harmonic
and $\alpha_{\lambda}$  the differential one-form defined above. Since
$\alpha_{\lambda}$ satisfies the integrability condition \eqref{integr},  we consider
the equation
\[\begin{split}\dd F(z,\lambda)&= F(z, \lambda)\alpha_{\lambda},\ F(z_0,\lambda)=\mathbf e,\\
\end{split}
\]
on any contractible open subset $U \subset M$,
where $z_0$ is a fixed base point on $U$, and $\mathbf e$ is the identity element in $G$. The map $F(z, \lambda)$
is called the {\em extended frame}
 of the harmonic map $f$ normalized at the base point $z=z_0$. Note that $F$ satisfies $F(z,\lambda =1)=F(z)$.
 \end{definition}


\subsubsection{Harmonic maps into $SO^+(1,n+3)/SO^+(1,3)\times SO(n)$}

 Let's consider the group   $SO^+(1,n+3)$ together with its Lie algebra
\begin{equation}\label{eq-so(1,n+3)}
\mathfrak{so}(1,n+3)=\mathfrak{g}=\{X\in gl(n+4,\mathbb{R})|X^tI_{1,n+3}+I_{1,n+3}X=0\}.
\end{equation}
Consider the involution
 \begin{equation}
\begin{array}{ll}
\sigma:  SO^+(1,n+3)& \rightarrow SO^+(1,n+3)\\
 \ \ \ \ \ \ \ A&\mapsto D^{-1} A D,
\end{array}\end{equation}
with
 $$D=\left(
         \begin{array}{ccccc}
             -I_{4} & 0 \\
            0 & I_n \\
         \end{array}
       \right),
       $$
       where $I_k$ denotes the $k\times k$ identity matrix. Then the fixed point group $SO^+(1,n+3)^{\sigma}$ of $\sigma$ contains $SO^+(1,3)\times SO(n)$,  where
$SO^+(1,3)$ denotes the connected component of $SO(1,3)$ containing $I$.
Moreover we have $SO^+(1,n+3)^{\sigma}\supset SO^+(1,3)\times SO(n) = (SO^+(1,n+3)^{\sigma})^0$,
where the superscript $0$ denotes  the connected component containing the identity element. On the Lie algebra level we obtain
   \begin{equation*}\begin{split}&\mathfrak{g}=
\left\{\left(
                   \begin{array}{cc}
                     A_1 & B_1 \\
                     -B_1^tI_{1,3} & A_2 \\
                   \end{array}
                 \right)
 |\ A_1^tI_{1,3}+I_{1,3}A_1=0, \ \ A_2+A_2^t=0\right\},\\
&\mathfrak{k}=\left\{\left(
                   \begin{array}{cc}
                     A_1 &0 \\
                     0 & A_2 \\
                   \end{array}
                 \right)
 | \ A_1^tI_{1,3}+I_{1,3}A_1=0,\ \ A_2+A_2^t=0\right\},\ \mathfrak{p}=\left\{\left(
                   \begin{array}{cc}
                   0 & B_1 \\
                     -B_1^tI_{1,3} & 0 \\
                   \end{array}
                 \right)
\right\}.
\end{split}
\end{equation*}

Now let $f: M\rightarrow SO^+(1,n+3)/SO^+(1,3)\times SO(n)$
 be a harmonic map with local frame $F: U\rightarrow SO^+(1,n+3)$ and Maurer-Cartan form $\alpha$ on some contractible open subset $U$ of $M$.
Let $z$ be a local complex coordinate on $U$. Writing
\begin{equation}\label{eq-B0}
    \alpha_{\mathfrak{k}}'=\left(
                   \begin{array}{cc}
                     A_1 &0 \\
                     0 & A_2 \\
                   \end{array}
                 \right) \dd z,\ \hbox{ and }\ \alpha_{\mathfrak{p}}'=\left(
                   \begin{array}{cc}
                     0 & B_1 \\
                     -B_1^tI_{1,3} & 0 \\
                   \end{array}\right)\dd z,\end{equation}

\begin{definition} \label{stronglyconfharm}
Let $f: M\rightarrow SO^+(1,n+3)/SO^+(1,3)\times SO(n)$ be a harmonic map. We call $f$ a {\bf strongly conformally harmonic map} if for any point $p\in M$, there exists a neighborhood $U_p$ of $p$ and a frame $F$ (with Maurer-Cartan form  $\alpha$) of $y$ on $U_p$ satisfying
\begin{equation}\label{eq-Willmore harmonic} B_1^t I_{1,3} B_1 = 0, ~ \hbox{where }~ \alpha_{\mathfrak{p}}'=\left(
                   \begin{array}{cc}
                     0 & B_1 \\
                     -B_1^tI_{1,3} & 0 \\
                   \end{array}\right)\dd z.\end{equation}
 \end{definition}

We call $f$ contains a constant lightlike vector, if $f(p)$, as subspaces of $\mathbb R^{n+4}_1$, contains a non-zero constant lightlike vector, for any $p\in M$. Note that the conform Gauss map of a minimal surface in $\mathbb R^{n+4}_1$ always contains a constant lightlike vector \cite{DoWa11}.
The main result of \cite{DoWa11} can be summarized as
\begin{theorem}\label{th-Willmore-harmonic}\cite{DoWa11}  Let $f: M\rightarrow SO^+(1,n+3)/SO^+(1,3)\times SO(n)$ be a non-constant strongly conformally harmonic map from a connected Riemann surface $M$. Then after changing the orientation of $M$ if necessary, then either $f$ contains a constant lightlike vector, or $f$ is the {oriented} conformal Gauss map of some unique conformal Willmore map $y:M\rightarrow S^{n+2}$.
\end{theorem}

\section{Loop group theory for harmonic maps}

In this section we start by collecting the basic definitions and the basic
decomposition theorems for loop groups (\cite{DPW}, \cite{Wu}, \cite{Ba-Do}).
Next we recall the DPW method for the construction of harmonic maps.
Since we are mainly interested in strongly conformally harmonic maps,
we characterize strongly conformal harmonicity in terms of normalized potentials which satisfy the additional
condition  $B_1^t I_{1,3} B_1 = 0$.


\subsection{Loop groups and decomposition theorems}

Let $G$ be a connected real Lie group and $G^\C$ its complexification
(For details on complexifications, in particular of semi-simple Lie groups, see \cite{Hochschild}).

Let $\sigma$ denote an inner involution of $G$ and $K$ a closed subgroup satisfying $(Fix^\sigma(G))^0 \subset K \subset Fix^\sigma(G)$.
Then $\sigma$ fixes $\mathfrak{k} = Lie K$.
The extension of $\sigma$  to an involution of $G^\C$ has
 $\mathfrak{k}^\C$  as  its fixed point algebra.
By abuse of notation we put $K^\C = Fix^\sigma(G^\C)$.
It is known that $K^\C$ is in general not connected.

Here are  the basic definitions about loop groups which we will apply to any  Lie group $G$ and its  complexification
$G^\C$, assuming an inner involution $\sigma$ of these groups is given:
 \begin{equation*}
\begin{array}{llll}
\Lambda G^{\mathbb{C}}_{\sigma} ~&=\{\gamma:S^1\rightarrow G^{\mathbb{C}}~|~ ,\
\sigma \gamma(\lambda)=\gamma(-\lambda),\lambda\in S^1  \},\\[1mm]

\Lambda G_{\sigma} ~&=\{\gamma\in \Lambda G^{\mathbb{C}}_{\sigma}
|~ \gamma(\lambda)\in G, \hbox{for all}\ \lambda\in S^1 \},\\[1mm]
\Omega G_{\sigma} ~&=\{\gamma\in \Lambda G_{\sigma}|~ \gamma(1)=\ee \},\\[1mm]

 \Lambda^{-} G^{\mathbb{C}}_{\sigma}  ~&=
\{\gamma\in \Lambda G^{\mathbb{C}}_{\sigma}~
|~ \gamma \hbox{ extends holomorphically to } |\lambda|>1 \cup\{\infty\} \},\\[1mm]

\Lambda_{*}^{-} G^{\mathbb{C}}_{\sigma} ~&=\{\gamma\in \Lambda G^{\mathbb{C}}_{\sigma}~
|~ \gamma \hbox{ extends holomorphically to } |\lambda|>1\cup\{\infty\},\  \gamma(\infty)=\ee \},\\[1mm]

\Lambda^{+} G^{\mathbb{C}}_{\sigma} ~&=\{\gamma\in \Lambda G^{\mathbb{C}}_{\sigma}~
|~ \gamma \hbox{ extends holomorphically to } |\lambda|<1 \},\\[1mm]

\Lambda_{L}^{+} G^{\mathbb{C}}_{\sigma} ~&=\{\gamma\in
\Lambda G^{\mathbb{C}}_{\sigma}~|~   \gamma(0)\in L \},\\[1mm]
\end{array}\end{equation*}
where $L$ denotes some Lie subgroup of $K^\C$, prescribing in which group the leading term is supposed to be contained in.
If $L = (K^\C)^0$, then we write $\Lambda_{\mathcal{C}}^{\pm} G^{\mathbb{C}}_{\sigma} $. These groups are connected.

For the decomposition theorems quoted below (which are of crucial importance for the applicability of the loop group method explained below) we need to have some topology on our loop groups.
This can be done in several ways.  We will represent $G^\C$ as a matrix group and will assume that all matrix entries
of  $\Lambda G^{\mathbb{C}}_{\sigma}$  are in the Wiener
algebra of the unit circle.  We thus obtain that $\Lambda G^{\mathbb{C}}_{\sigma}$ is a Banach Lie group.
All other groups discussed in this paper inherit a Banach Lie group structure in a natural way.

With these assumptions we obtain:
\begin{theorem} {\em(Birkhoff Decomposition Theorem for $ (\Lambda {G}_\sigma^\C )^0$)} \label{thm-birkhoff-0}
\begin{enumerate}
\item $ (\Lambda {G}_\sigma^\C )^0= \bigcup \Lambda^{-}_{\mathcal{C}} {G}^{\mathbb{C}}_{\sigma} \cdot \omega \cdot \Lambda^{+}_{\mathcal{C}} {G}^{\mathbb{C}}_{\sigma}$,
where the $\omega$'s are representatives of the double cosets.

\item The multiplication $\Lambda_{*}^{-} {G}^{\mathbb{C}}_{\sigma}\times
\Lambda^{+}_\FC {G}^{\mathbb{C}}_{\sigma}\rightarrow
\Lambda {G}^{\mathbb{C}}_{\sigma}$ is an analytic  diffeomorphism onto the
open and dense subset $\Lambda_{*}^{-} {G}^{\mathbb{C}}_{\sigma}\cdot
\Lambda^{+}_\FC {G}^{\mathbb{C}}_{\sigma}$  of $ (\Lambda G^\C_\sigma)^0$
{\em (big Birkhoff cell)}.
\end{enumerate}
\end {theorem}

\begin{remark} \
\begin{enumerate}
\item For the case of Willmore surfaces we consider  the inner symmetric space $G/K$, where $G = SO^+(1,n+3)$ and $K = SO^+(1,3) \times SO(n)$.
\item The inner involution  $\sigma$ is given by $\sigma = Ad I_{4,n}$ with $ I_{4,n} = diag(I_4, -I_n)$. Then $ Fix^\sigma(G) = S( O^+(1,3) \times O(n))$ and $K$ is the connected component of this fixed point group.

\item For the complexification  of $G$ we obtain
 $G^\C \cong SO(1,n+3,\C)$ and the fixed point group of $\sigma$ in $G^\C$ (called $K^\C$,  by abuse of notation, as above) is $K^\C \cong Fix^\sigma(G^\C) = S(O(1,3,\C) \times O(n,\C))$.  Moreover, we obtain that $(K^\C)^0$ is the complexification of $({Fix}^{\sigma}(G))^0$.

\item For the simply-connected covers $\tilde{G}$ and $\tilde{G}^\C$ of $G$ and  $ G^\C$ respectively we obtain
$\tilde{G} = Spin(1,n+3)^0$ and
$\tilde{G}^\C = Spin(1, n+3,\C)$  (See e.g. \cite{LM}, $(2.35)$  and \cite{Mein}, Proposition 3.1 respectively).

\end{enumerate}
\end{remark}

The discussion so far is related to part one of the loop group method (as outlined in detail below).
Part two of the loop group method requires another decomposition theorem:

\begin{theorem} {\em (Iwasawa Decomposition Theorem for
$(\Lambda G^{\C})_{\sigma} ^0 $\label{gen-Iwasawa})} \label{thm-Iwasawa-0}
\begin{enumerate}
\item
$(\Lambda G^{\C})_{\sigma} ^0=
\bigcup \Lambda G_{\sigma}^0\cdot \delta\cdot
\Lambda_\mathcal{C}^{+} G^{\mathbb{C}}_{\sigma},$
where the $\delta$'s are representatives of the double cosets.

 \item The multiplication $\Lambda G_{\sigma}^0 \times \Lambda_\mathcal{C}^{+} G^{\mathbb{C}}_{\sigma}\rightarrow
(\Lambda G^{\mathbb{C}}_{\sigma})^0$ is a real analytic map onto the connected open subset
$ \Lambda G_{\sigma}^0 \cdot \Lambda_\mathcal{C}^{+} G^{\mathbb{C}}_{\sigma}   = \mathcal{I}^{\mathcal{U}}_e \subset \Lambda G^{\mathbb{C}}_{\sigma}$.
\end{enumerate}

\end{theorem}

For the loop groups involved in the description of Willmore surfaces we add two results:

\begin{theorem}
Consider the setting $G/K = Gr_{1,3}(\mathbb{R}^{n+4}_{1}) = SO^+(1,n+3)/SO^+(1,3)\times SO(n)$.
\begin{enumerate}
\item
 There exist  two different open Iwasawa cells in the connected loop group
$(\Lambda G^{\mathbb{C}}_{\sigma})^0$,  one given by $\delta = e$ and the other one by $\delta = diag(-1,1,1,1,-1,1,1,...,1) $.

\item
There exists a closed, connected, solvable subgroup
$S \subseteq (K^\C)^0$ such that
the multiplication

$\Lambda G_{\sigma}^0 \times \Lambda^{+}_S G^{\mathbb{C}}_{\sigma}\rightarrow
(\Lambda G^{\mathbb{C}}_{\sigma})^0$ is a real analytic diffeomorphism onto the connected open subset
$ \Lambda G_{\sigma}^0 \cdot \Lambda^{+}_S G^{\mathbb{C}}_{\sigma}      \subset  \mathcal{I}^{\mathcal{U}}_e \subset(\Lambda G^{\mathbb{C}}_{\sigma})^0$.
\end{enumerate}
\end{theorem}

 Proofs for the last three theorems are given in  Appendix A.


\subsection{The DPW method and potentials }

 With the loop group decompositions as stated above, we obtain a
construction scheme of harmonic maps from a surface into any real pseudo-Riemannian symmetric space $G/K$ for which the metric is induced from a bi-invariant metric on $G$. All symmetric spaces considered in this paper are of this type.

So far we have mainly discussed Willmore surfaces and the corresponding conformally harmonic maps
defined on some open subset $U$ of $\C$ (or possibly an open subset of some surface $M$).
Since the immersions of interest are conformal, the corresponding surface has a complex structure.
We thus only consider  Riemann surfaces in this paper.
 If $M$ is such a Riemann surface, then its universal cover
$\tilde{M}$ is either $S^2$ or $\C$ or $\E$, the open unit disk in $\C$.
Every harmonic map from $M$ to some  symmetric space $G/K$ induces via composition with the natural projection
a harmonic map from the universal cover $\tilde{M}$ into $G/K$.
Therefore, to start with, we need to consider harmonic maps from $S^2$, $\C$ and $\E$ into $G/K$.

\begin{theorem}{\em(\cite{DPW})}\label{DPW}
Let $\D$ be a contractible open subset of $\C$ and $z_0 \in \D$ a base point.
Let $f: \D \rightarrow G/K$ be a harmonic map with $f(z_0)=\ee K.$
Then the associated family  $f_{\lambda}$ of $f$ can be lifted to a map
$F:\D \rightarrow \Lambda G_{\sigma}$, the extended frame of $f,$ and we can assume  w.l.g. that $F(z_0,\lambda)= \ee$ holds.
Under this assumption,
\begin{enumerate}
\item
  The map $F$ takes only values in
$ \mathcal{I}^{\mathcal{U}}_e \subset \Lambda G^{\mathbb{C}}_{\sigma}$.

\item There exists a discrete subset $\D_0\subset \D$ such that on $\D\setminus \D_0$
we have the decomposition
\[F(z,\lambda)=F_-(z,\lambda)\cdot F_+(z,\lambda),\]
where
\[F_-(z,\lambda)\in\Lambda_{*}^{-} G^{\mathbb{C}}_{\sigma}
\hspace{2mm} \mbox{and} \hspace{2mm} F_+(z,\lambda)\in (\Lambda^{+} G^{\mathbb{C}}_{\sigma})^0.\]
Moreover $F_-(z,\lambda)$ is meromorphic in $z \in \D$ and $F_-(z_0,\lambda)=\ee$ holds and the Maurer-Cartan form $\eta$ of $F_-$
\[\eta= F_-(z,\lambda)^{-1} \dd F_-(z,\lambda)\]
is a $\lambda^{-1}\cdot\mathfrak{p}^{\mathbb{C}}-\hbox{valued}$ meromorphic $(1,0)-$
form with poles at points of $\D_0$ only.

\item Conversely, any harmonic map  $f: \D\rightarrow G/K$ can be derived from a
$\lambda^{-1}\cdot\mathfrak{p}^{\mathbb{C}}-\hbox{valued}$ meromorphic $(1,0)-$ form $\eta$ on $\D$.

\item  Spelling out the converse procedure in detail we obtain:
Let $\eta$ be a  $\lambda^{-1}\cdot\mathfrak{p}^{\mathbb{C}}-\hbox{valued}$ meromorphic $(1,0)-$ form for which the solution
to the ODE
\begin{equation}
F_-(z,\lambda)^{-1} \dd F_-(z,\lambda)=\eta, \hspace{5mm} F_-(z_0,\lambda)=\ee,
\end{equation}
is meromorphic on $\D$, with  $\D_0$ as set of possible poles.
Then on the open set $\D_{\mathcal{I}} = \lbrace z \in \D; F(z,\lambda)
\in \mathcal{I}^{\mathcal{U}} \rbrace$ we
define $\tilde{F}(z,\lambda)$ via the factorization
 $\mathcal{I}^{\mathcal{U}}_e =  ( \Lambda G_{\sigma})^0 \cdot \Lambda_S^{+} G^{\mathbb{C}}_{\sigma}
\subset  \Lambda G^{\mathbb{C}}_{\sigma}$:
\begin {equation}\label{Iwa}
F_-(z,\lambda)=\tilde{F}(z,\lambda)\cdot \tilde{F}_+(z,\lambda)^{-1}.
\end{equation}
 This way one obtains an extended frame
$$ \tilde{F}(z,\lambda)=F_-(z,\lambda)\cdot \tilde{F}_+(z,\lambda)\\$$
of some harmonic map from $\D_{\mathcal{I}} \subset \D$ to $G/K$ satisfying  $\tilde{F}(z_0,\lambda)=\ee$.

Moreover, the two constructions outlined above  are inverse to each other (on appropriate domains of definition).
\end{enumerate}
\end{theorem}

\begin{definition}{\em(\cite{DPW})}
The $\lambda^{-1}\cdot \mathfrak{p}^{\mathbb{C}}-\hbox{valued}$ meromorphic $(1,0)$
form $\eta$ is called the {\em normalized potential} for the harmonic
map $f$ with the point $z_0$ as the reference point.
\end{definition}

\begin{remark}\
\begin{enumerate}
\item
Note that the normalized potential is uniquely determined once a base point is chosen.
However, if we conjugate a normalized potential by  some $z-$independent element  $k$ of $K$,
then the procedure outlined in the theorem produces a new harmonic map (and correspondingly a new Willmore surface)
which differs from the original one by the rigid motion induced by $k$.
Since we usually do not care about how the harmonic map (or the Willmore surface) sits
in  space, we sometimes use elements of $K$ to simplify or further normalize the normalized potential.
 \item
In the converse procedure, part $(4)$ above, since in our case the symmetric space $G/K$ is not compact,
the Iwasawa splitting \eqref{Iwa} will in general not be possible for all $z \in \D$.
Thus $\tilde{F}$, as well as the harmonic map $\tilde{f}$ will have singularities on $\D$. There are two types of singularities. One type stems from poles in the potential $\eta$ and the other type occurs,
when $F_-$   touches or crosses the boundary of an  open Iwasawa cell (See \cite{B-R-S}).
(There are at least two open Iwasawa cells, as pointed out above).
In the new example of a Willmore sphere in $S^6$  \cite{Wang-iso} (also see Example \ref{thm-example} mentioned in the introduction)  it happens
that the frame of the harmonic map has  singularities, but the Willmore immersion and hence the harmonic map does not have any singularity
 (\cite{Wang-iso}).  In this case the frame is the product of a matrix without singularities with a singular scalar factor. The projection of the frame to the harmonic map cancels out this singular factor. So the appearance of singularities is only due to the choice of frame.

\end{enumerate}\end{remark}

So far we have only introduced the ``normalized potential". However, in many applications it is much more convenient to use potentials which contain, in their Fourier expansion, more than one power of $\lambda$.
The normalized potential is usually meromorphic in $z$. Since it is uniquely determined, there is no way to change this. However, when permitting many (maybe infinitely many) powers of $\lambda$, then one can obtain potentials with holomorphic coefficients,  which will be called {\em holomorphic potentials}. They are not uniquely determined.

\begin{theorem} \label{pot-hol} Let $\D$ be a contractible open subset of $\C$.
Let $F(z,\lambda)$ be the frame of some harmonic map
into $G/K$. Then there exists some $V_+ \in \Lambda^{+} G^{\mathbb{C}}_{\sigma} $ such that $C(z,\lambda) =
F V_+$ is holomorphic in $z$ and in $\lambda \in \mathbb{C}^*$.
Then the Maurer-Cartan form $\eta = C^{-1} dC$ of $C$ is a holomorphic $(1,0)-$form on $\D$ and it is easy to verify that $\lambda \eta$ is holomorphic for $\lambda \in \C$.
Conversely, any harmonic map  $f: \D\rightarrow G/K$ can be derived from
such a holomorphic $(1,0)-$ form $\eta$ on $\D$ by the same steps as in the previous theorem.
\end{theorem}
The proof can be taken verbatim from the appendix of
\cite{DPW} and will  be omitted here.
 We would like to point out that the proof works for all harmonic maps into any Riemannian or pseudo-Riemannian symmetric space (actually, even more generally, for primitive harmonic maps into $k-$symmetric spaces).
In particular, the proof is independent of the results of the previous sections.

\begin{remark}\
\begin{enumerate}
 \item
Again, since the Iwasawa splitting is not global in our case, even when starting with a holomorphic potential, the resulting harmonic map will generally have singularities.

\item
Let $\eta_1$ and $\eta_2$ be any two potentials producing the same harmonic map by the procedure outlined above. Then there exists a gauge
$W_+:\D \rightarrow\Lambda^+G^{\C}_{\sigma}$ transforming one potential into the other. For a proof consider the frames $F_1 = C_1 V_{+1}$ and
$F_2 = C_2 V_{+2}$ constructed as outlined above. Since we assume that the two potentials induce the same harmonic map, these frames only differ by some gauge: $F_1 = F_2 T$ where $T \in K$. This implies
$C_1 V_{+1} = C_2 V_{+2} T$. Thus $W_+ = V_{+2} T V_{+1}^{-1}$ is the desired gauge.
\end{enumerate}\end{remark}
\vspace{1mm}

So far we have only discussed potentials for harmonic maps defined on some
contractible open subset of $\C$. Let now $M$ denote a Riemann surface which is either non-compact or compact of positive genus.
Then the universal cover $\tilde{M}$ of $M$ can be realized as a contractible open subset of $\C$. Moreover, if $f:M \rightarrow G/K$ is a harmonic map, then the composition $\tilde{f}$ of $f$ with the canonical projection from $\tilde{M}$ onto $M$ is also harmonic. Therefore to $\tilde{f}$ we can construct normalized potentials and holomorphic potentials as outlined above. These potentials for $\tilde{f}$
will also be called potentials for $f$.
The converse procedure as outlined in the last two theorems produces harmonic maps defined on some open subsets (containing the base point) of $\D$.
For these harmonic maps to descend to $M$ ``closing conditions" need to be satisfied.

\begin{remark}\
\begin{enumerate}
 \item If $M = S^2$, then it is not clear a priori that the procedure discussed in Theorem \ref{DPW} works as well. However, if the  symmetric target space actually is a real Lie group $G,$ considered as a symmetric space $G \cong (G \times G)/\Delta$, where $\Delta$ denotes the subgroup
$\Delta = \lbrace (g,g) \in G \times G, g \in G \rbrace$ and one uses the natural projection $(g,h) \rightarrow  gh^{-1},$ then the same procedure works.
In this case one can  lift a harmonic map $f: S^2 \rightarrow G$  to $G \times G$
by $F = (f,e)$. This way one obtains, as in the previous cases, a normalized potential. It has the form
$\xi = (\lambda^{-1} \eta , -\lambda^{-1} \eta)$. Harmonic maps into Lie groups (as symmetric spaces) have been discussed in \cite{Do-Es}, Section 9
. Note, however, that the formula given in \cite{Do-Es} for the normalized potential shows a wrong $\lambda-$ dependence.

 \item On the other hand, one does not obtain a holomorphic potential for $M=S^2$, since $S^2$ does not carry any non-trivial holomorphic $(1,0)-$forms. The proof of \cite{DPW} which was  used in the proof of the theorem above  is not applicable to the case $M = S^2$, of course.
\end{enumerate}\end{remark}
Let's consider now the case that we have a harmonic map $f$ from $M = S^2$
into some general symmetric space $G/K$. Since $\pi_2 (S^2) = \mathbb{Z}$ and  $\pi_2 (G) = 0$, it will generally not be possible to lift the smooth map $f$ to a smooth map $F$  from $S^2$ to $G$ such that the map $F$ composed with the natural projection from $G$ to $G/K$ is $f$. But one can find some way around this non-lifting obstacle.
\begin{theorem}\label{th-potential-sphere}
Every harmonic map from $S^2$ to any  Riemannian or pseudo-Riemannian symmetric space $G/K$ admits an extended frame with at most two singularities and it admits  a global meromorphic extended frame. In particular, every harmonic
 map from $S^2$ to any  Riemannian or pseudo-Riemannian symmetric space $G/K$
can be obtained from some meromorphic normalized potential.
\end{theorem}

\begin{proof}Let $f:S^2\rightarrow G/K$ be a harmonic map. Set $$\mathcal{U}_1=S^2\backslash \{\hbox{ north pole }\},
\ ~\mathcal{U}_2=S^2\backslash \{\hbox{ south pole }\}$$
and $f_1=f|_{\mathcal{U}_1}$, $f_2=f|_{\mathcal{U}_2}$. Since $\mathcal{U}_1\cong\mathcal{U}_2\cong\mathbb{C}$, there exist frame lifts $F_j:\mathcal{U}_j\rightarrow G$ of $f_j$, $j=1,2$, by \cite{DPW}.

We can assume w.l.g. $F_1(p_0)=F_2(p_0)=e$ where $p_0$ is a fixed base point in $\mathcal{U}_1\cap\mathcal{U}_2$ and $f(p_0)=e \mod K$.
Also we have $F_2=F_1\mathcal{K}$ on $\mathcal{U}_1\cap\mathcal{U}_2$. Introducing $\lambda$ yields ($\sigma-$twisted) $F_1$ and $F_2$ and again
$F_2=F_1\mathcal{K}$, where $F_j=F_j(z,\bar{z},\lambda)$ and $\mathcal{K}=\mathcal{K}(z,\bar{z})$. By \cite{DPW}, there exist discrete subsets $D_j\subset \mathcal{U}_j$, $j=1,2$ such that
$$F_j=F_{j-}F_{j+}, \ ~~j=1,2$$
on $\mathcal{U}_j\backslash D_j$. Moreover, $F_{j-}$ extends to a meromorphic map on $\mathcal{U}_j$ by \cite{DPW}.

On $(\mathcal{U}_1\cap\mathcal{U}_2)\backslash (D_1\cup D_2)$ we have
$F_{2-}V_{2+}=F_{1-}V_{1+}\mathcal{K}$, where $F_{j-}=I+\mathcal{O}(\lambda^{-1})$.
Hence
$$F_{2-}=F_{1-} \ ~\hbox{ on }\ ~ (\mathcal{U}_1\cap\mathcal{U}_2)\backslash (D_1\cup D_2).$$
Therefore this meromorphic map on $\mathcal{U}_1\cap\mathcal{U}_2$ extends meromorphically to $S^2$.
Set $\eta=F_-^{-1}\dd F_-.$
Then $\eta$ is a meromorphic $(1,0)-$form on $S^2$ of the form $\eta=\lambda^{-1}\eta_{-1}\dd z.$   As usual, $\eta$ will be called ``normalized potential'' of $f$. Moreover, by reversing the steps above we see that the map $f$ can be obtained from $\eta$ as usual, however,  we need to admit (up to two) singularities for the  extended  frame defined by $\eta$. Thus $\eta$ is justifiedly called the  normalized potential of $f$.
\end{proof}

\begin{remark}\
 \begin{enumerate}
  \item By removing just one point of $S^2$, like the north pole, one obtains a meromorphic map on $\mathcal{U}_1$ which, however, could have an essential singularity at the north pole. The use of $\mathcal{U}_1$ and $\mathcal{U}_2$ as above shows that $F_{1-}=F_{2-}$ on $\mathcal{U}_1\cap\mathcal{U}_2$ actually extends as a meromorphic frame to all points of $S^2$.
 \item From the proof above it is clear that the original harmonic map $f$ can be  reconstructed  from $\eta$ by the usual steps  (see Theorem \ref{DPW}). Note that the two procedures just discussed are inverse to each other.

 \item The theorem just proven can be used to construct all harmonic maps from $S^2$ into any Riemannian or pseudo-Riemannian symmetric space:

\emph{ Consider a meromorphic $(1,0)-$form on $S^2$ of the form $\eta=\lambda^{-1}\eta_{-1}\dd z $ which has a meromorphic solution $F_-$ on $S^2$ to the ode \[ F_- \eta=\dd F_-, F_-(p_0,\lambda) = e.\]}

Now an Iwasawa decomposition of $F_-$ makes sense for all points, where $F_-$ is in the open Iwasawa cell containing $e$, producing a ``frame" $F$ which is an actual frame on the set of non-singular points of $F_-$. Let $\tau$ denote the anti-holomorphic involution of $G^{\C}$  defining $G$.
Since $F$ is obtained via a Birkhoff decomposition of $\tau(F_{-})^{-1}F_-$ in the form
$\tau(F_{-})^{-1}F_-=\tau(V_{+})V_+^{-1}$, we obtain  $F=F_-V_+=\tau(F_-V_+)$, and its matrix entries are rational functions in the entries of $\tau(F_{-})^{-1}F_-$. In particular, the matrix entries of $F$ are rational functions in $u,v$, $z=u+iv$. Now a harmonic map is obtained by $f = F \mod K$.

 \item Since for pseudo-Riemannian spaces the Iwasawa splitting is not global in general, not every $\eta$ as above will yield a singularity free harmonic map on all of $S^2$.
The domain of definition of $f$ will need to be discussed separately in each case.
\end{enumerate}
\end{remark}

\begin{corollary} Let $f: S^2\rightarrow G/K$ be a harmonic map and $\eta$  its normalized potential with reference point $z_0$. Then away from the (finitely many) poles of $\eta$ there exists an extended frame $F$ for $f$ and a global Iwasawa splitting $F=F_-F_+$, $F_-^{-1}\frac{\dd}{\dd z}F_-=\eta$.
Moreover, $F_-$ is meromorphic on $S^2$.
\end{corollary}

Clearly, the normalized potential just discussed lives on $S^2 = M$.
 For general surfaces $M$ the potentials for harmonic maps only live on the universal cover $\tilde{M}$ of $M$. Therefore, if one wants to construct harmonic maps from some arbitrary Riemann surface $M$ into some  symmetric space $G/K$, one would have at least some indication for where to find an appropriate potential, if one would know that for every harmonic map from $M$ to  $G/K$ there is some potential  on $\tilde{M}$ which is the pullback of some differential one-form defined on $M$.
So far there is known \cite{Do-Ha5}, Theorem 3.2

\begin{theorem}
If $M$ is non-compact, then for every harmonic map from $M$ to any
 Riemannian or pseudo-Riemannian
symmetric space there exists a holomorphic potential defined  on the universal cover
$\tilde{M}$ of $M$ which is invariant under the fundamental group of $M$.
\end{theorem}
 \begin{remark}\
\begin{enumerate}
 \item
 By abuse of notation we sometimes say in the situation described above that the potential is defined on $M$.

\item For the case of compact surfaces $M$ we conjecture

\vspace{1.5mm}
{ \it Every harmonic map from any compact Riemann surface $M$ to any pseudo-Riemannian symmetric space
can be obtained from some meromorphic potential defined on $M$.}\vspace{1.5mm}\\
 In \cite{DoWa-sym1} we will prove this conjecture for all compact Riemann surfaces and for the  pseudo-Riemannian symmetric space occurring in our Willmore setting.
\end{enumerate}
\end{remark}

\subsection{The normalized potential for strongly conformally harmonic maps
and Wu's formula}

From the definition of the normalized potential (see Theorem \ref{DPW}) we can read off that
it is obtained from the $\lambda^{-1}-$part of the Maurer-Cartan form
of $F$ by conjugation by some matrix function with values in $K^\mathbb{C}$.
For known examples one can write down the normalized potential much more specifically.
In \cite{Wu}, Wu showed how one can determine locally the normalized potential
from the Maurer-Cartan form of the harmonic map $f$.

In this subsection we will make this explicit for the case of primary interest to this paper. As an immediate consequence of Theorem \ref{DPW} we obtain:

\begin{theorem} \label{normalized-potential-W} Let $\D$ be a contractible open subset of $\C$ and $0 \in \D$ a base point.

 Let $f: \D\rightarrow SO^+(1,n+3)/SO^+(1,3)\times SO(n)$ be a strongly
conformally harmonic map with $f(0)=eK$ and
$F:\D \rightarrow (\Lambda G_{\sigma})^0$ an extended frame of $f$
such that $F(0,\lambda) = I$. Then the normalized potential of $f$ with respect to the base point $z = 0$ is of the form
\begin{equation}\label{eq-potential-W}
\eta= \lambda^{-1}\eta_{-1}\dd z,\  \hbox{ with }\ \eta_{-1}=\left(
    \begin{array}{cc}
      0 & \hat{B}_1 \\
      -\hat{B}_1^tI_{1,3} & 0 \\
    \end{array}
  \right)\dd z,\hspace{3mm} \mbox{with} \hspace{3mm} \hat{B_1}^tI_{1,3}\hat{B}_1=0,
\end{equation}
where $\hat{B}_1\dd z$ is a meromorphic $(1,0)-$ form on $\D$ and $0$ is not a pole of $\hat{B}_1$.

Conversely,  any  normalized potential defined on $\D$ and satisfying  (\ref{eq-potential-W}) induces a strongly conformally harmonic
map from an open subset $0 \in \D_\mathcal{I} \subset \D$ into $SO^+(1,n+3)/SO^+(1,3)\times SO(n)$.
\end{theorem}

\begin{remark}
Using Theorem \ref{th-potential-sphere} one can formulate an analogous
result for $\D$ replaced by $S^2$.
\end{remark}

Similarly we obtain as an immediate consequence of Theorem \ref{pot-hol}:

\begin{theorem}\label{Holo-Willmore}
Let $\D$ be a contractible open subset of $\C$ and $0 \in \D$ a base point.
 Let $f: \D\rightarrow SO^+(1,n+3)/SO^+(1,3)\times SO(n)$ be a strongly
conformally harmonic map with $f(0)=eK$ and
$F:\D \rightarrow (\Lambda G_{\sigma})^0$ an extended frame of $f$
such that $F(0,\lambda) = I$. Then there exists a holomorphic potential for $f$ and each holomorphic potential for $f$ is of the form
\begin{equation} \label{eq-potential-H}
\xi=(\lambda^{-1}\xi_{-1}+\sum_{j\geq0}\lambda^j\xi_j)\dd z,\  \hbox{ with }\ \xi_{-1}=\left(
    \begin{array}{cc}
      0 & \hat{B}_1 \\
      -\hat{B}_1^tI_{1,3} & 0 \\
    \end{array}
  \right),\hspace{3mm} \mbox{and} \hspace{3mm} \hat{B_1}^tI_{1,3}\hat{B}_1=0,
\end{equation}
where $\xi_{j}\dd z$, $j=-1,0,\cdots,\infty$, are holomorphic $(1,0)-$ forms on $\D$.

Conversely,  any holomorphic potential $\eta$ defined on $\D$ and satisfying \eqref{eq-potential-H} induces  a strongly conformally harmonic
map from an open subset $0 \in \D_\mathcal{I} \subset \D$ into $SO^+(1,n+3)/SO^+(1,3)\times SO(n)$.
\end{theorem}

The matrix function $\hat{B_1}$ in the previous theorem (normalized potential) can be made much more explicit.

\begin{theorem} \label{Wu-W} ( Wu's Formula for Strongly Conformally Harmonic Maps)

Let $\D$ be a contractible open subset of $\C$ and $0 \in \D$ a base point.

 Let $f: \D\rightarrow SO^+(1,n+3)/SO^+(1,3)\times SO(n)$ be a strongly
conformally harmonic map with $f(0)=eK$ and
$F:\D \rightarrow (\Lambda G_{\sigma})^0$ an extended frame of $f$
such that $F(0,\lambda) = I$. Consider  $ F^{-1} \dd F = \alpha = \lambda^{-1} \alpha'_{\mathfrak{p}} + \alpha_{\mathfrak{k} }+ \lambda \alpha^{\prime \prime}_{\mathfrak{p}}$ and let $\delta_1$ denote the sum of the
holomorphic terms in the Taylor expansion of $\alpha'_{\mathfrak{p}}(\frac{\partial}{\partial z})$ about $0$, considered
as a form depending on $z$ and $\bar{z}$. The form $\delta_1$ is called the {\em holomorphic part}
of $\alpha'_{\mathfrak{p}}(\frac{\partial}{\partial z})$. Similarly, denote by $\delta_0$ the holomorphic part of $\alpha'_{\mathfrak{k}}(\frac{\partial}{\partial z})$.

Then the  normalized potential $\eta$ of $f$ with the origin as the reference point is given by
\begin{equation}
\eta=\lambda^{-1}\eta_{-1}\dd z\ \hbox{ with }\eta_{-1}=F_0(z)\delta_1F_0(z)^{-1},
\end{equation}
where $F_0: \D \rightarrow  G/K$ is  the solution to the equation $F_0(z)^{-1}\dd F_0(z)=\delta_0,\  F_0 (0) = I.$
\end{theorem}
\begin{proof}
 By \eqref{eq-MC-1} and \eqref{eq-b1}( See also Theorem 3.4 of \cite{DoWa11}), we know  the form of
$\alpha'_{\mathfrak{p}}(\frac{\partial}{\partial z}).$
Since the holomorphic part  $\delta_1$ of  $\alpha'_{\mathfrak{p}}(\frac{\partial}{\partial z})$ has the same form, we obtain by setting $\tilde{B}_1 (z, \bar{z} = 0)$
\[\delta_1=
\left(
    \begin{array}{cc}
      0 & \widetilde{B}_1 \\
      -{\widetilde{B} }_1^tI_{1,3} & 0 \\
    \end{array}
  \right)\dd z,\ \hbox{ with }\ {\widetilde{B} }_1^tI_{1,3}{\widetilde{B} }_1=0.
\]
Let
$F_0=\hbox{diag}(\hat{A}_1,\hat{A}_2):\D \rightarrow  SO^+(1,3,\mathbb{C})\times SO(n,\mathbb{C})$ be the solution to the equation
$F_0(z)^{-1}\dd F_0(z)=\delta_0,$ $F_0(0)=I,$ where $\delta_0$ is the holomorphic part of $\alpha'_{\mathfrak{k}}(\frac{\partial}{\partial z})$.
Then Wu's Formula \cite{Wu} implies for the normalized potential
\[\eta_{-1}=F_0(z)\delta_1 F_0(z)^{-1}=\lambda^{-1} \left(
    \begin{array}{cc}
      0 & \hat{A}_1\widetilde{B} _1\hat{A}_2^{-1} \\
      -\hat{A}_2{\widetilde{B} }_1^tI_{1,3}\hat{A}_1^{-1} & 0 \\
    \end{array}
  \right)\dd z=\lambda^{-1}\left(
    \begin{array}{cc}
      0 & \hat{B}_1 \\
      -\hat{B}_1^tI_{1,3} & 0 \\
    \end{array}
  \right)\dd z.\]
Moreover,  from ${\widetilde{B} }_1^tI_{1,3}{\widetilde{B} }_1=0$ we obtain $\hat{B}_1^tI_{1,3}\hat{B}_1=0$.
\end{proof}
\begin{remark}
\
\begin{enumerate}
\item Note that one can assume w.l.g. that $\alpha$ has the special form stated in
Section 2.1 for the MC form of the conformal Gauss maps. However this will not imply in general that $\eta_{-1}$ has such a special form. Later, in Section 4, we will show that only very special harmonic maps admit such kinds of normalized potentials.

\item
It is straightforward to verify that $\eta_{-1}$ in \eqref{eq-potential-W} satisfies
$$\eta_{-1}^3=0.$$ So $\eta$ is pointwise nilpotent as a Lie algebra-valued function.
However this does not imply that $\eta$ attains all values in a fixed nilpotent Lie subalgebra.
As a consequence, in general the corresponding conformally harmonic map is not of finite uniton type.
An example for this is the Clifford torus in $S^3$, which is of finite type and not of finite uniton type.

 \item In the last theorem we have considered local expansions of real analytic functions into power series in $z$ and $\bar{z}$ about $z=0$ and set $\bar{z} = 0$. So the factors entering into the formula for $\eta$ will in general  only be defined locally. However, $\eta$ itself is defined and meromorphic globally on $\D$.

If one wants to find globally defined factors for the representation of $\eta$ above, then one needs to analyze the proof of the corresponding result
of \cite{Do-Ko}.
\end{enumerate}
\end{remark}



 \section{Application of Loop group theory to Willmore surfaces}

In this section we will present applications of Wu's formula for two types of harmonic maps.


\subsection{Strongly conformally harmonic maps containing a constant light-like vector}

From Theorem \ref{th-Willmore-harmonic}, we see that there are two kinds of conformally harmonic maps satisfying $B_1^tI_{1,3}B_1=0$:
those which contain a constant lightlike vector and those which do not contain a constant lightlike vector.
Moreover, if a conformally harmonic map $f$ does not contain a lightlike vector, then $f$ will always be the conformal Gauss map of some Willmore map.
This class of Willmore maps corresponds exactly to all those Willmore maps which are {\em not }conformal to any minimal surface in $\mathbb{R}^{n+2}$,
since minimal surfaces in $\mathbb{R}^{n+2}$ can be characterized as Willmore surfaces with their conformal Gauss map containing a constant lightlike vector.
Since minimal surfaces in  $\mathbb{R}^{n+2}$ can be constructed by direct methods,
we are mainly interested in Willmore surfaces not conformally equivalent to minimal surfaces in  $\mathbb{R}^{n+2}$.
It is therefore vital to derive a criterion to determine whether a strongly conformally harmonic map $f$ contains a lightlike vector or not.
This is the main goal of this subsection. We state the main result and refer for a proof (which uses substantially the techniques discussed in the previous sections)  to \cite{Wang-Min}.

\begin{theorem}\cite{Wang-Min} \label{th-potential-light}Let $\tilde{M}$ denote the Riemann surface $S^2, \mathbb{C}$ or the unit disk of $\mathbb{C}$. Let $f:  \mathbb{D}\rightarrow SO^+(1,n+3)/SO^+(1,3)\times SO(n)$ be a strongly conformally harmonic map which contains a constant light-like vector.  Choose a base point $p \in \tilde{M}$ and assume that $f(p)=I_{n+4}\cdot K$ holds. Let  $z$ denote a local coordinate with $z(p)=0$. Then the normalized potential of $f$ with reference point $p$ is of the form
 \begin{equation}\label{eq-w-minimal}
\eta=\lambda^{-1}\left(
                     \begin{array}{cc}
                       0 & \hat{B}_1 \\
                       -\hat{B}^{t}_1I_{1,3} & 0 \\
                     \end{array}
                   \right)\dd z,\ \hbox{ where }\ \hat{B}_{1}=\left(
\begin{array}{cccc}
 \hat{f}_{11} & \hat{f}_{12} & \cdots &  \hat{f}_{1n} \\
 -\hat{f}_{11} &  -\hat{f}_{12} & \cdots &  -\hat{f}_{1n} \\
 \hat{f}_{31} &\hat{f}_{32} & \cdots &  \hat{f}_{3n} \\
 i\hat{f}_{31} & i\hat{f}_{32} & \cdots &  i\hat{f}_{3n} \\
 \end{array}
\right).\end{equation}
Here all  $f_{ij}$ are meromorphic functions on $\tilde{M}$.

The converse  also holds:
Let  $\eta$ be a normalized potential of the form \eqref{eq-w-minimal}. Then $B_1^tI_{1,3}B_1=0$ and we obtain a strongly conformally harmonic map $f: \tilde{M}\rightarrow SO^+(1,n+3)/SO^+(1,3)\times SO(n)$. Moreover, $f$ contains a constant light-like vector and is of finite uniton type.
\end{theorem}

\begin{remark} \
The proof of this result  requires a lengthy argument and will therefore be published in \cite{Wang-Min}.
It is not difficult to verify that $f$ is of finite uniton type.
Moreover, $f$ actually belongs to  the simplest case of finite uniton maps, the so-called $S^1-invariant$ maps (See \cite{BuGu}, \cite{Do-Es}).
For such harmonic maps, by a usually lengthy computation, one can derive the harmonic map directly without using loop groups,
since the Iwasawa splitting in this case is identical with the classical generalized Iwasawa splitting for non-compact Lie groups (see \cite{Do-Es}).
\end{remark}

\begin{corollary}
Let $f: \tilde{M}\rightarrow SO^+(1,n+3)/SO^+(1,3)\times SO(n)$ be a strongly conformally harmonic map
with its normalized potential $\eta$ of the form \eqref{eq-w-minimal} and of maximal $rank(\hat B_1)=2$. Then $f$ can not be
the conformal Gauss map of a Willmore surface. In particular, there exist strongly conformally harmonic maps which are not related to any Willmore map.
\end{corollary}
This is a straightforward application of Theorem 3.11 of \cite{DoWa11} and Theorem \ref{th-potential-light}.

  Using the loop group method it is easy to see that harmonic maps satisfying the assumptions of the corollary always exist at least locally.
Moreover, combining Theorem Theorem 3.11 of \cite{DoWa11} and  Theorem \ref{th-potential-light}, and applying Wu's formula, it is straightforward to obtain the following
\begin{corollary} Let $f: \tilde{M}\rightarrow SO^+(1,n+3)/SO^+(1,3)\times SO(n)$ be a strongly conformally harmonic map
with its normalized potential $\eta$ of the form \eqref{eq-w-minimal} and of maximal $rank(\hat B_1)=1$. Then $f$ can not be
the conformal Gauss map of a Willmore surface if and only if up to a conjugation, $\hat B_1$ has one of the following forms
\begin{equation}\label{eq-w-min-2}
   \hat{B}_{1}=\left(
\begin{array}{cccc}
 0 & 0& \cdots &  0\\
 0 & 0& \cdots &  0\\
 \hat{f}_{31} &\hat{f}_{32} & \cdots &  \hat{f}_{3n} \\
 i\hat{f}_{31} & i\hat{f}_{32} & \cdots &  i\hat{f}_{3n} \\
 \end{array}\right), \hbox{ or } \hat{B}_{1}=\left(
\begin{array}{cccc}
 \hat{f}_{11} & \hat{f}_{12} & \cdots &  \hat{f}_{1n} \\
 -\hat{f}_{11} &  -\hat{f}_{12} & \cdots &  -\hat{f}_{1n} \\
 0 & 0& \cdots &  0\\
 0 & 0& \cdots &  0\\
 \end{array}\right).
\end{equation}
\end{corollary}
\begin{proof}  If $f$ can not be
the conformal Gauss map of a Willmore surface, then   by Theorem 3.11 of \cite{DoWa11},
it reduces to a harmonic map into $SO^+(1,n+1)/SO^+(1,1)\times SO(n)$ or $SO(n+2)/SO(2)\times SO(n)$. As a consequence, applying Wu's formula we see that the normalized potential reduces to $\Lambda \mathfrak{so}(1,n+1,\C)_{\sigma}$ or $\Lambda \mathfrak{so}(n+2,\C)_{\sigma}$. Now \eqref{eq-w-min-2} follows.
 The converse part is also straightforward. Since $\eta$ has the form stated in  \eqref{eq-w-min-2}, clearly  $f$  reduces to $SO^+(1,n+1)/SO^+(1,1)\times SO(n)$ or $SO(n+2)/SO(2)\times SO(n)$.
\end{proof}

\subsection{The conformal Gauss map of isotropic Willmore surfaces in $S^4$}

Another important class of Willmore surfaces is formed by the {\em totally isotropic} Willmore surfaces.
We recall from section 2 that $D$ denotes the $V_\C^\perp-$part
of the natural connection of $\C^4$. By $D_z^j$ we denote the $j-$fold iteration of $D_z$.

\begin{definition}(\cite{Ca}, \cite{Bryant1982}, \cite{Ejiri1988}) Let $y:M\rightarrow S^{n+2}$ be a conformal immersion with  $z$ a local coordinate of $M$ and $Y$ a  local lift. Then $y$ is called {\em totally isotropic} if the Hopf differential $\kappa$ of $y$ satisfies
 \begin{equation} \langle D_{z}^{j}\kappa,  D_{z}^{l}\kappa \rangle = 0, \hbox{ for } j,\ l=0,\ 1,\cdots.\end{equation}
\end{definition}
 Note that full  and totally isotropic surfaces only exist  in even dimensional spheres $S^{2m}$. They can be described as projections of holomorphic (anti-holomorphic) curves in the twistor bundle $\mathfrak{T}S^{2m}$ of $ S^{2m}$ (\cite{Ca}, \cite{Ejiri1988}).

However, in general, totally isotropic surfaces in $S^{2m}$ are not necessarily Willmore surfaces when $m>2$. Thus totally isotropic Willmore surfaces in $S^{2m}$ are of particular interest.
A much larger class of surfaces is formed by the isotropic surfaces, i.e. the surfaces, where in the definition above only the case $j = l = 0$ is required. Fairly little is known  about general isotropic  surfaces.

However, it is well-known that  all isotropic surfaces in $S^4$ are  Willmore surfaces (even S-Willmore surfaces), see \cite{Ejiri1988}, \cite{Ma2006}.

In this subsection we will characterize all isotropic Willmore surfaces in $S^4$.  An analysis of isotropic Willmore surfaces in $S^6$ will be presented in \cite{Wang-iso}. Concerning isotropic (Willmore) surfaces in $S^4$, we show
\begin{theorem}\label{th-potential-iso}
Let $y:M\rightarrow S^4$ be an isotropic surface from a simply connected Riemann surface $\tilde{M}$, with its conformal Gauss map $f=Gr$ defined in Section 2. Then the normalized potential of $Gr$ is of the form
\begin{equation}\label{eq-isotropic}\eta=\lambda^{-1}\left(
                     \begin{array}{cc}
                       0 & \hat{B}_1 \\
                       -\hat{B}^{t}_1I_{1,3} & 0 \\
                     \end{array}
                   \right)\dd z, \hbox{ with }
\hat{B}_{1}=\left(
                                                                       \begin{array}{cccc}
                                                                         \hat{f}_{11} & i\hat{f}_{11}   \\
                                                                          \hat{f}_{21} & i \hat{f}_{21}   \\
                                                                          \hat{f}_{31} &i\hat{f}_{31}   \\
                                                                         \hat{f}_{41} & i\hat{f}_{41}  \\
                                                                       \end{array}
                                                                     \right),\ -\hat{f}_{11} ^2+\hat{f}_{21} ^2+\hat{f}_{31} ^2+\hat{f}_{41} ^2=0.
\end{equation}
Moreover, $Gr$ is of finite uniton type with uniton number $r(f)$ at most 2. In particular, $f$ is $S^1$-invariant.

Conversely, let $\eta$ be defined on $\tilde{M}$ of the form \eqref{eq-isotropic} and let $f: \tilde{M}\rightarrow SO^+(1,5)/SO^+(1,3)\times SO(2)$ be the associated strongly conformally harmonic map. Then either $f$ is the conformal Gauss map of an isotropic S-Willmore surface in $S^{4}$, or $f$ takes values in   $SO^+(1,3)/SO^+(1,1)\times SO(2)$ or in $SO(4)/SO(2)\times SO(2)$ and is not the conformal Gauss map of any conformal immersion.
\end{theorem}
\begin{proof} Retaining the notation of Section 2.1  for $y$ and $Gr$, the isotropy property of $y$ shows that $\langle\kappa,\kappa\rangle=0$. Differentiating  this expression for $z$ one obtains
$\langle D_{\bar{z}}\kappa,\kappa\rangle=0.$
Noticing that the complex normal bundle is 2 dimensional and observing  that $\kappa$ is a  null vector section, it is clear that any other section perpendicular to $\kappa$  is necessarily  parallel to $\kappa$.  Hence we infer  that $D_{\bar{z}}\kappa$ is parallel to $\kappa$.
So without loss of generality, we can assume
\[\kappa=k_1\psi_1+ik_1\psi_2, \hbox{ and } D_{\bar{z}}\kappa=\beta_1\psi_1+i\beta_1\psi_2,\]
with $\psi_1,$ $\psi_2$ an orthonormal basis of sections of $V^{\perp}$ in the sense of Section 2.
Therefore the Maurer-Cartan form of $F(z,\bar{z},\lambda)$ w.r.t $y$ is\\
\[ F^{-1}F_z=\left(
                   \begin{array}{cc}
                     A_1 & B_1 \\
                     -B_1^tI_{1,3} & A_2 \\
                   \end{array}
                 \right),\
\hbox{ with }\
B_1=\left(
                                                                       \begin{array}{cccc}
                                                                          \sqrt{2}\beta_1 &  i\sqrt{2}\beta_1  \\
                                                                         -\sqrt{2}\beta_1 & -i\sqrt{2}\beta_1  \\
                                                                         -k_{1} & -ik_{1}   \\
                                                                         -ik_{1} & -ik_{1}  \\
                                                                       \end{array}
                                                                     \right).\]
To apply Wu's formula  (Theorem \ref{Wu-W}),
let  $\delta_{1}$, $\delta_{2}$ and  $\tilde{B}_1$  denote the ``holomorphic parts" of $ A_{1}$, $A_2$ and $B_1$ with respect to the reference point $z=0$ respectively, i.e., the part of the Taylor expansion of $A_1$, $A_2$ and $B_1$ which are independent of $\bar{z}$.
Let $F_{01}$ and $F_{02}$ be the solutions to the equations $F_{01}^{-1}\dd F_{01}=\delta_{1}\dd z,\  F_{01}|_{z=0}=I_4$ and
$F_{02}^{-1}\dd F_{02}=\delta_{2}\dd z,\  F_{02}|_{z=0}=I_2$ respectively.
By Wu's formula (Theorem \ref{Wu-W}), the normalized potential can be represented in the form
\[\eta=\lambda^{-1}\left(
                     \begin{array}{cc}
                       0 & \hat{B}_1 \\
                       -\hat{B}^{t}_1I_{1,3} & 0 \\
                     \end{array}
                   \right)\dd z,\ ~~\hbox{ with }\ ~~\hat{B}_{1}=F_{01}\tilde{B}_{1}F_{02}^{-1}.
\]
Noticing that here $\tilde{B}_1$ is of the form $\tilde{B}_1=\left(v_1, iv_1\right),\ \hbox{ with } \  v_1\in \mathbb{C}^4_1,$  it is immediate
to check that $F_{01}\tilde{B}_1 F_{02}^{-1}=\left(\hat{v}_1, i\hat{v}_1\right)$
holds with some vector $\hat{v}_1 \in \C^4$.

A straightforward computation shows that $\hat{B}_1^t I_{1,3}\hat{B}_1=0$ is equivalent with $  \hat{v}_1^tI_{1,3}\hat{v}_1=0$, and
\eqref{eq-isotropic} follows now.
In view of the definition of harmonic maps of finite uniton type \cite{BuGu,DoWa2,Uh}, the last statement is a corollary to the fact that $\eta(\frac{\partial}{\partial z})$ in \eqref{eq-isotropic} takes values in a nilpotent Lie subalgebra of
 degree of nilpotency $2$,  which shows that $F_-$ will be a polynomial in $\lambda^{-1}$ of degree at most $2$.

As to the converse part, assume that $\hat B_1=(\hat{v}_1,i\hat{v}_1)$. Let $f$ be the corresponding harmonic map with $B_1=(v_1,v_2)$. Then we have $B_1=V_{01}\hat{B}_{1}V_{02}^{-1}$ for some $V_{01}\in SO^+(1,3,\C)$ and $V_{02}\in SO(2,\C)$. So we have $v_2=iv_1$ holds. In particular $rank B_1\leq 1$. Applying Theorem 3.11 of \cite{DoWa11}, we see the Theorem holds except the isotropic property. But this is a consequence of the facts that $B_1$ being of the form $(v_1,iv_1)$ is independent of the choice of frames, and $B_1$ being of the form $(v_1,iv_1)$ is equivalent to the condition that the corresponding Willmore surface is isotropic (Note that these facts only hold in the case of codimension $2$).
\end{proof}

\begin{remark} \
\begin{enumerate}
\item Isotropic surfaces in $S^4$ provide another type of strongly conformally harmonic maps of finite uniton number $\leq2$,
which actually have an intersection with minimal surfaces in $\mathbb{R}^4$ (see e.g. the examples below). For more details, we refer to \cite{Mon}.
 And also note that  a Weierstrass type representation for isotropic minimal surfaces in $S^4$  has been presented in \cite{Bryant1982}.

 \item The case of isotropic Willmore surfaces in $S^6$ shows a very different situation, in particular by the fact that, in general, they can not be of finite uniton type \cite{Wang-iso}.
\end{enumerate}\end{remark}

By the classification theorems in \cite{Ejiri1988}, \cite{Mus1}, \cite{Mon},  a Willmore two-sphere in $S^4$ is either isotropic or is conformally equivalent to a minimal surface in $\mathbb{R}^4$. Applying Theorem \ref{th-potential-light} and Theorem \ref{th-potential-iso}, we obtain
\begin{corollary}The conformal Gauss map of a Willmore two-sphere in $S^4$ is of finite uniton type with $r(y)\leq2$ and hence is $S^1$-invariant.
\end{corollary}

 A main result of \cite{Le-Pe-Pin} states that a Willmore torus in $S^4$ with non-trivial normal bundle is either isotropic or is conformally equivalent to a minimal surface in $\mathbb{R}^4$.
Together with the above results we derive the following
\begin{corollary}The conformal Gauss map of a Willmore torus in $S^4$ with non-trivial normal bundle is of finite uniton number at most 2 and hence $S^1$-invariant.
\end{corollary}

We can also state  the classification theorem of Bohle \cite{Bo} on Willmore tori in $S^4$ as below. For the notion of ``finite type'' we refer to \cite{Bo}.

\begin{corollary}\cite{Bo} The conformal Gauss map of a Willmore torus in $S^4$ is either of finite type or of finite uniton number at most 2 (and hence is $S^1$-invariant in the latter case).
\end{corollary}


\subsection{On homogeneous Willmore surfaces admitting an Abelian group action}

\begin{definition}
A Willmore  immersion $y: M \rightarrow S^{n+2}$ is called homogeneous if there exists a group
$\Gamma:=\{(\gamma,R_{\gamma}):\ \gamma\in Aut(M),\ R\in SO^+(1,n+3)\}$ such that
\begin{equation}y(\gamma\cdot z)=R_{\gamma}\cdot y(z), \hbox{ for all } z\in M \hbox{ and } \ (\gamma,R_{\gamma})\in \Gamma
\end{equation}
and  the projection $\Gamma_M$ of $\Gamma$ onto the first factor acts transitively on $M$.
\end{definition}

 Clearly, with  $\Gamma$ also the closure
$\overline{\Gamma}$ in  $Aut(M)  \times SO^+(1, n+3)$ satisfies the conditions of the definition. We can thus assume that $\Gamma$ is a Lie group. Of course, the closure
$\overline{\Gamma_M}$ in $Aut(M)$ is  transitive on $M$.

In \cite{DoWaHom} we classifiy all homogeneous Willmore surfaces in spheres.
Here we describe how one can construct explicitly all homogeneous Willmore surfaces which
have an abelian transitive group action.

\begin{theorem}
 Let $y:M \rightarrow S^{n+2}$ be a homogeneous Willmore immersion from a Riemann surface $M$.
Assume that the group $\Gamma$ is abelian. Then
\begin{enumerate}
\item  $M = \C$ and $\Gamma_M \cong$  all translations; or $M =S^1\times \R$ and
$\Gamma_M \cong S^1\times \mathbb{Z}$; or $M = \mathbb{T}$ and
$\Gamma_M \cong \mathbb{Z}^2$.

 \item  For the lift $\tilde{y}$ of $y$ to the universal cover $\C$ of $M$ there exists an  extended frame associated with the conformal Gauss map of $\tilde{y}$, which has a  constant
 Maurer-Cartan form $\alpha$ of the form stated in Section 2.1, \eqref{eq-MC-1}--\eqref{eq-b1}.
      Moreover, setting $\eta=\alpha'= \lambda^{-1} \eta_{-1} + \eta_0$, then $\eta$ is a constant, real, holomorphic potential, which generates the immersion
    $\tilde{y}$ (and hence also $y$)  and satisfies $\lbrack \eta \wedge \bar{\eta} \rbrack = 0$ and $\eta_{013} + \eta_{023} \neq 0$. Here $\eta_{013}$ and $\eta_{023}$ denote the (1,3) and (2,3) term of  $\eta_0(\frac{\partial}{\partial z})$

 \item Conversely, let $\eta$ be a constant real potential,  defined on $\C$ and  of the form $\eta = \lambda^{-1} \eta_{-1} + \eta_0$ satisfying $\lbrack \eta \wedge \bar{\eta} \rbrack = 0 $ and having the same form as $\alpha'$ in Section 2.1, \eqref{eq-MC-1}--\eqref{eq-b1}. Then it generates a homogeneous Willmore immersion (without branch points) for which its conformal Gauss map has an extended frame with constant Maurer-Cartan form.
\end{enumerate}
\end{theorem}
\begin{proof} (1). Since $M$ is a connected Riemann surface with  $Aut(M)$ containing
a two-dimensional abelian Lie subgroup, the universal covering of $M$ is $\C$ and the rest follows straightforwardly.

 (2). Choose the coordinate on $\C$,  we see that the group $\Gamma$ consists of all translations. Hence we obtain
for the extended frame of the conformal Gauss map the relation \[F(u+ iv, u-iv,\lambda) =
exp(u\mathfrak{X}) exp(vi\mathfrak{Y}) F(0,0,\lambda),\] where $\mathfrak{X}$ and $\mathfrak{Y}$ commute and only depend on $\lambda$.
Therefore the Maurer-Cartan form of $F$ is constant and
we obtain \[F^{-1} \dd F = \mathfrak{X}\dd u + \mathfrak{Y}\dd v = (\mathfrak{X}+\mathfrak{Y})\dd z + i(\mathfrak{X}-\mathfrak{Y}) \dd\bar{z}.\]
So $\mathfrak{X}-\mathfrak{Y}$ only involves non-negative powers of $\lambda$ and $\mathfrak{X}+\mathfrak{Y}$ only non-positive powers of lambda. Moreover, the
matrices $\mathfrak{X}-\mathfrak{Y}$ and $\mathfrak{X}+\mathfrak{Y}$ commute.
As a consequence, assuming w.l.g. $F(0,0,\lambda) = I_{n+4}$, we also obtain
\[F(u+ iv, u-iv,\lambda) =
exp(u\mathfrak{X}) exp(v \mathfrak{Y})  = \exp (z (\mathfrak{X}+\mathfrak{Y})) \exp( \bar{z} i(\mathfrak{X}-\mathfrak{Y})).\]
Since this is some Birkhoff  decomposition of $\exp (z (\mathfrak{X}+\mathfrak{Y}))$
we conclude that \[ \eta=\exp (z (\mathfrak{X}+\mathfrak{Y}))^{-1}\dd\exp (z (\mathfrak{X}+\mathfrak{Y}))=(\mathfrak{X}+\mathfrak{Y})\dd z =(\lambda^{-1}\eta_{-1}+\eta_0)\dd z\] is a holomorphic potential for the given immersion of the type stated.

(3). Assume now $\eta$ is of the special form stated, then
$\eta=(\lambda^{-1}\mathfrak{B}+\mathfrak{A})\dd z$ with $\mathfrak{A},$ $\mathfrak{B}$ constant matrices satisfying
\[[\lambda^{-1}\mathfrak{B}+\mathfrak{A},\lambda\bar{\mathfrak{B}}+\bar{\mathfrak{A}}]=0.\]
Then
\[e^{z(\lambda^{-1}\mathfrak{B}+\mathfrak{A})}=e^{z(\lambda^{-1}\mathfrak{B}+\mathfrak{A})+\bar{z}(\lambda\bar{\mathfrak{B}}+\bar{\mathfrak{A}})}\cdot e^{-\bar{z}(\lambda\bar{\mathfrak{B}}+\bar{\mathfrak{A}})}\]
is an Iwasawa decomposition, producing the extended frame
\[F(z,\bar{z},\lambda)=e^{z(\lambda^{-1}\mathfrak{B}+\mathfrak{A})+\bar{z}(\lambda\bar{\mathfrak{B}}+\bar{\mathfrak{A}})}.\]
This implies that the conformally harmonic map $f = F \mod K$ is conformally homogeneous.
Since the Maurer-Cartan form of $F(z,\bar{z},\lambda)=(e_0,\hat e_{0}, e_1,e_2,\psi_1,\cdots,\psi_n)$ is of the form stated in Section 2.1, $F$ is the conformal Gauss map of some immersion $y=[e_0-\hat{e}_0]$. The harmonicity of the conformal Gauss map indicates that $y$ is a Willmore immersion.
\end{proof}

\begin{corollary} Every homogeneous Willmore torus in $S^{n+2}$ can be obtained from a constant potential of the form
\begin{equation}\label{eq-homo-int}
  \eta = \lambda^{-1} \eta_{-1} + \eta_0,~~\hbox{ with }~[\eta_{-1},\overline{\eta_0}]=0,\  [\eta_{-1},\overline{\eta_{-1}}]+ [\eta_{0},\overline{\eta_0}]=0,
\end{equation}
and $ \eta_{-1}$, $ \eta_{0}$ being of the same form  as $ \alpha_{\mathfrak{p}}'(\frac{\partial}{\partial z})$ and $\alpha_{\mathfrak{k}}'(\frac{\partial}{\partial z})$ respectively in \eqref{eq-MC-1}, \eqref{eq-MC-2}, \eqref{eq-MC-3} and \eqref{eq-b1} ( See also Theorem 2.2 of \cite{DoWa11}).
\end{corollary}

A special case of homogeneous strongly conformally harmonic maps is produced by ``{\bf vacuum potentials}". Recall the definition of a vacuum potential \cite{BP}
$$\eta=(\lambda^{-1}\mathfrak{B})\dd z,\ \hbox{ with } [\mathfrak{B},\bar{\mathfrak{B}}]=0.$$
Such a potential always produces a harmonic map $f$. For $f$ being a strongly conformally harmonic map, one needs to assume that
$$\mathfrak{B}=\left(
                   \begin{array}{cc}
                    0 & B_1 \\
                    - B_1^tI_{1,3} & 0 \\
                   \end{array}
                 \right),\ \hbox{ with }\ B_1^tI_{1,3}B_1=0.$$
Then, as shown in Lemma 3.3 of \cite{DoWa11}, there exists some $L_1\in SO^+(1,3)$ such that $L_1B_1$ is of the form \eqref{eq-b1}.
By Theorem \ref{th-potential-light}, one sees that for $f$ being the conformal Gauss map of some Willmore map $y$, the maximal rank of $B_1$ must be one.
Hence we may assume that
$$B_1=(\mathrm{v}_1,\cdots,\mathrm{v}_n) \ ~\hbox{ with }\ ~\mathrm{v}_j=(a_j+ib_j) \mathrm{v}_0,\ a_j, b_j\in\R,\ j=1,\cdots,n. $$
Here $\mathrm{v}_0\in \hbox{Span}_{\C}\{(1,-1,0,0)^t,(0,0,1,i)^t\}.$
If $\langle\mathrm{v}_0,\mathrm{v}_0\rangle=0$, then $\mathrm{v}_0\in \hbox{Span}_{\C}\{(1,-1,0,0)^t\}$.
So we see that $f$ reduces to a harmonic map into $SO(1,n)/SO(1,1)\times SO(n)$. If  $\langle\mathrm{v}_0,\mathrm{v}_0\rangle\neq0$, there exists another
$L_2\in SO^+(1,3)$ such that $L_2\mathrm{v}_0\in \hbox{Span}_{\C}\{(0,0,1,i)^t\}$.
As a consequence,  $f$ reduces to a harmonic map into $SO(n+2)/SO(2)\times SO(n)$.

In a sum, we obtain
\begin{proposition}Let $f$ be a vacuum solution which is also a strongly conformally harmonic map. Then it can not be the conformal Gauss map of any Willmore surface.
\end{proposition}

\begin{example}Let $y=[Y]: S^1\times R^1\rightarrow S^4$ be the cylinder
  \begin{equation}
Y=\left( \cosh av, \sinh av, \cos u \cos bv, \cos u \sin bv, \sin u \cos bv, \sin u \sin bv \right)^t
\end{equation}
with $a^2+b^2=1,\ a,b\in \mathbb{R}$. Note, if $a=0$ we obtain the Clifford torus in $S^3\subset S^4$, and if $b=0$ we obtain the round sphere with the north pole removed. (For a detailed discussion on Willmore tori in $S^4$, we refer to \cite{Bo} and \cite{Le-Pe-Pin}.)

A direct computation shows that $y$ is a homogeneous Willmore immersion, with a holomorphic potential
 \begin{equation}
\tilde\eta=\left(
                              \begin{array}{cccccc}
                                0 & 0 & \frac{1}{4\sqrt{2}} & \frac{-i(1+2a^2)}{4\sqrt{2}} & \frac{iab\lambda^{-1}}{2\sqrt{2}} & 0 \\
                                0 & 0 & \frac{3}{4\sqrt{2}} & \frac{-i(1+2b^2)}{4\sqrt{2}} & \frac{-iab\lambda^{-1}}{2\sqrt{2}} & 0 \\
                                \frac{1}{4\sqrt{2}} & \frac{-3}{4\sqrt{2}} & 0 & 0 & 0  & \frac{ ib \lambda^{-1}}{2}\\
                                \frac{-i(1+2a^2)}{4\sqrt{2}} & \frac{i(1+2b^2)}{4\sqrt{2}} & 0 & 0 & 0 &  \frac{ b\lambda^{-1}}{2}\\
                                 \frac{iab\lambda^{-1}}{2\sqrt{2}} & \frac{iab\lambda^{-1}}{2\sqrt{2}} & 0 & 0 & 0 & -\frac{a}{2}\\
                                0  & 0 & \frac{-ib\lambda^{-1}}{2} & \frac{-b\lambda^{-1}}{2}  & \frac{a}{2} & 0\\
                              \end{array}
                            \right)\dd z.
\end{equation}
\end{example}
\begin{example}\cite{Wang-Homo}  Let $y:\C \rightarrow S^5(\sqrt{1+2b^2})$
  \begin{equation}
y= \left(\cos  u \cos \frac{v}{\sqrt{3}}, \cos u \sin \frac{v}{\sqrt{3}}, \sin  u \cos \frac{v}{\sqrt{3}}, \sin u \sin \frac{v}{\sqrt{3}}, \sqrt{2}b\cos \frac{v}{\sqrt{3}b}, \sqrt{2}b\sin \frac{v}{\sqrt{3}b}\right)^t
\end{equation}
with $b\in \mathbb{R}^+$. Obviously $y$ is a torus if and only if $b\in \mathbb{Q}^+$.
A direct computation shows that $y$ is a homogeneous Willmore immersion, with a holomorphic potential
 \begin{equation}
\tilde\eta=\left(
                              \begin{array}{ccccccc}
                                0 & 0 & s_1 & s_2 & 0& 0 & \lambda^{-1}\sqrt{2}\beta_3 \\
                                0 & 0 & s_3 & s_4 & 0& 0 & -\lambda^{-1}\sqrt{2}\beta_3 \\
                               s_1& -s_3 & 0 & 0 & -\lambda^{-1}k_1 & -\lambda^{-1}k_2 & 0\\
                               s_2& -s_4 & 0 & 0 & -\lambda^{-1}ik_1 & -\lambda^{-1}ik_2 & 0\\
                               0& 0 & \lambda^{-1}k_1 & \lambda^{-1}ik_1 &0 & 0 & -a_{13}\\
                               0& 0& \lambda^{-1}k_2 & \lambda^{-1}ik_2 & 0 & 0 & -a_{23}\\
                               \lambda^{-1}\sqrt{2}\beta_3 & \lambda^{-1}\sqrt{2}\beta_3 & 0 & 0 &a_{13}  & a_{23} & 0\\
                              \end{array}
                            \right)\dd z,
\end{equation}
with
\[k_1=\frac{\sqrt{4b^2+2}}{12b},k_2=\frac{-i\sqrt{3}}{6}, s=\frac{4b^2-1}{18 b^2}, \beta_3=\frac{-i\sqrt{2}(4b^2-1)}{72b^2}, a_{13}=\frac{-i\sqrt{2b^2+1}}{6b}, a_{23}=\frac{\sqrt{6}}{6},\]
and
\[s_1=\frac{\sqrt{2}(20b^2+1)}{144b^2},\ s_2=\frac{-i\sqrt{2}(12b^2-1)}{48b^2},\ s_3=\frac{\sqrt{2}(52b^2-1)}{144b^2}, s_2=\frac{-i\sqrt{2}(12b^2+1)}{48b^2}.\]
Note that in this case one obtains the Ejiri's torus when $b=1$ \cite{Ejiri1982}.

Moreover, if we assume that $b=\frac{j}{l}$ with $j,l\in \mathbb{Z}^+$ and $(j,l)=1$, we obtain a torus with period $2\pi (1+il\sqrt{3})$. So the corresponding Willmore functional is
\[W(\mathbb{T}^2)=4\int_{0}^{2\pi}\dd u\int_{0}^{2\pi l\sqrt{3}}\dd v\left(|k_1|^2+|k_2|^2\right)=\frac{16\pi^2\sqrt{3}}{9}\left(l+\frac{j^2}{8l}\right).\]
\end{example}
\ \\

\section{Appendix: Two Decomposition Theorems}

In this section we discuss the basic decomposition theorems for loop groups.
We will use the notation introduced in Section 3.
 Since the decomposition theorems usually are proven for loops in simply-connected groups we will assume in this section that $G$ is simply-connected and will therefore always use, to avoid confusion,  the notation $\tilde{G}$. We would like to point out that the case of the group $G = SL(2,\R)$ is included in our presentation, but needs, at places, some interpretation, since in this case $\tilde{G}$ is not a matrix group, while $\tilde{G}^\C = SL(2,\C)$ is a matrix group and ``contains''  $\tilde{G}$ as a (non-isomorphic) image of the natural homomorphism.


\subsection{Birkhoff Decomposition}

 Starting from $\tilde{G}$ and an inner involution $\sigma$,  there is a unique extension, denoted again by $\sigma$, to $\tilde{G}^\C$.
The corresponding fixed point subgroups will be denoted by $\tilde{K}^\C$. Note that the latter group is connected, by a result of Springer$-$Steinberg.

Next we will consider the twisted loop group
$\Lambda \tilde{G}^{\C}_{\sigma} $.
General loop group theory implies that, since we  consider inner involutions only, we have  $\Lambda \tilde{G}^{\C}_{\sigma}  \cong
\Lambda \tilde{G}^{\C}$.
On the other hand,  we know $\pi_0 (\Lambda H) = \pi_1 (H)$ for any connected Lie group $H$, whence we infer that
$\Lambda \tilde{G}^{\C}_{\sigma} $ is connected. And since $\tilde{K}^\C$ is connected, also the groups $\Lambda^{+} \tilde{G}^{\mathbb{C}}_{\sigma}$
and $\Lambda^{-} \tilde{G}^{\mathbb{C}}_{\sigma}$ are connected.

For the loop group method used in this paper two decomposition theorems are of crucial importance. The first is

\begin{theorem}\label{thm-birkhoff}(Birkhoff decomposition theorem) Let $\tilde{G}^{\mathbb{C}}$ denote a simply-connected complex Lie group
with connected real form $\tilde{G}$ and let $\sigma$ be an inner involution of $\tilde{G}$ and  $\tilde{G}^{\mathbb{C}}$. Then the following statements hold
\begin{enumerate}
\item  $\Lambda \tilde{G}^{\C}_{\sigma} =
\bigcup \Lambda^{-} \tilde{G}^{\mathbb{C}}_{\sigma} \cdot \tilde{\omega} \cdot \Lambda^{+} \tilde{G}^{\mathbb{C}}_{\sigma},$
where the $\tilde{\omega}$'s are representatives of the double cosets.\\

\item The multiplication
 \begin{equation}
\Lambda^-_* \tilde{G}^{\mathbb{C}}_{\sigma} \times\Lambda^+ \tilde{G}^{\mathbb{C}}_{\sigma}\rightarrow
\Lambda^-_* \tilde{G}^{\mathbb{C}}_{\sigma} \cdot\Lambda^+ \tilde{G}^{\mathbb{C}}_{\sigma}
\end{equation}
is a complex analytic diffeomorphism and the (left) ``big cell"
$ \Lambda^-_* \tilde{G}^{\mathbb{C}}_{\sigma}
\cdot\Lambda^+ \tilde{G}^{\mathbb{C}}_{\sigma}$ is open and dense in
$ \Lambda  \tilde{G}^{\mathbb{C}}_{\sigma}$.\\

\item More precisely, every $g$  in $ \Lambda  \tilde{G}^{\mathbb{C}}_{\sigma}$ can be written in the form
 \begin{equation}g=g_-\cdot \tilde{\omega}\cdot g_+\end{equation}
with $g_{\pm}\in \Lambda^{\pm}  \tilde{G}^{\mathbb{C}}_{\sigma}$,
 and $\tilde{\omega} : S^1 \rightarrow  \tilde{T} \subset Fix^\sigma (\tilde{G}^\C)$ a homomorphism , where $\tilde{T}$ is a maximal compact torus in $\tilde{G}^\C$ fixed pointwise by $\sigma$.\\
\end{enumerate}
\end{theorem}

\begin{proof} The decomposition above has been proven for algebraic loop groups in \cite{Ka-P}.
Our results follow by completeness in the Wiener Topology (see e.g. \cite{DPW}).
\end{proof}

\begin{remark}
\begin{enumerate}
\item Our actual goal is to obtain a Birkhoff decomposition theorem for $(\Lambda G_\sigma^\C)^0$. The restriction to the connected component
is possible, since we will always consider maps from connected surfaces into the loop group which attain the value $I$ at some point of the surface.

This is very fortunate: since we have obtained above  a Birkhoff decomposition for the simply connected  complexified universal group, we will attempt to obtain the desired Birkhoff decomposition by projection.
Since $\Lambda \tilde{G}^{\C}_{\sigma} $ is connected,
applying the natural extension of the natural projection  $\tilde{ \pi}^\C : \tilde{G}^{\mathbb{C}} \rightarrow {G}^{\mathbb{C}}$  to $\Lambda \tilde{G}^{\mathbb{C}}_{\sigma} $ we obtain as image  the connected component  $(\Lambda G_\sigma^\C)^0$ of $\Lambda G_{\sigma}^\C$ .
We thus obtain the desired Birkhoff decomposition by projecting the terms on the right side. But since the groups  $\Lambda^{+} \tilde{G}^{\mathbb{C}}_{\sigma}$
and $\Lambda^{-} \tilde{G}^{\mathbb{C}}_{\sigma}$ are connected their images under the projection are $\Lambda^{+}_\FC {G}^{\mathbb{C}}_{\sigma}$
and $\Lambda^{-}_\FC {G}^{\mathbb{C}}_{\sigma}$ respectively.
 From this the Birkhoff Decomposition Theorem for $(\Lambda G_\sigma^\C)^0$ follows. The special case of primary interest in this paper will be discussed in detail in the following remark.

\item  Let $J$ denote a nondegenerate quadratic form in $\R^{n+4}$
and  $SO(J,\C)$
the corresponding real special orthogonal group. Let $SO(J,\C)$ denote the
complexified special orthogonal group. Then $SO(J,\C)$ is connected and
has fundamental group $\pi_{1}(SO(J,\C))\cong \mathbb{Z}/2\mathbb{Z}$.
Moreover, if $\sigma$ is an involutive inner automorphism, we have
$\Lambda SO(J,\C)\cong \Lambda SO(J,\C)_{\sigma}$. Therefore
 \begin{equation}
\pi_0 ( \Lambda SO(J,\C)_{\sigma}) \cong\pi_0 ( \Lambda SO(J,\C))\cong \pi_{1}(SO(J,\C))\cong \mathbb{Z}/2\mathbb{Z}.
\end{equation}
The loop group $\Lambda SO^+(1,\tilde{n},\mathbb{C})_{\sigma}$ thus has two connected components.
Finally, choosing $\sigma, K$ and $K^\C$ in the Willmore setting, the group $K^\C$ has two connected components. Therefore also
$\Lambda^+ {G}^{\mathbb{C}}_{\sigma}$ and
$\Lambda^{-} {G}^{\mathbb{C}}_{\sigma}$ have two connected components.

\item Much of the above is contained in \cite{PS}, Section 8.5 (see also \cite{SW}). Note, however, that our real group $G=SO^+(1,n+3)$ is not compact.
\end{enumerate}
\end{remark}

\ \\
{\em Proof of Theorem \ref{thm-birkhoff-0}.}\

  Consider the universal cover $\pi: Spin(1,n+3, \mathbb{C})\rightarrow SO(1,n+3,\mathbb{C})$.
Then $\pi$ induces a homomorphism from  $\Lambda Spin(1,n+3, \mathbb{C})_\sigma$ to
$(\Lambda SO(1,n+3, \mathbb{C})_\sigma)^0$, the identity component of  $\Lambda SO(1,n+3, \mathbb{C})_\sigma$.
 Projecting the decomposition of  Theorem \ref{thm-birkhoff} with $\tilde G^\C = Spin(1,n+3,\C)$ to $ SO(1,n+3, \mathbb{C})$, we obtain the Birkhoff factorization Theorem \ref{thm-birkhoff-0}.
\hfill $\Box$


\subsection{Iwasawa Decomposition}

From here on we will write, for convenience, $\Lambda G^0_\sigma$ for
$(\Lambda G_\sigma)^0$.
For our geometric applications we also need a second loop group decomposition. Ideally we would like to be able to write
any $g\in (\Lambda G^\mathbb{C})^0_\sigma$ in the form $g=hv_+$ with  $h\in (\Lambda G)^0_\sigma$ and
 $ v_+ \in (\Lambda^+ G^\mathbb{C})^0_\sigma =
\Lambda^+_\mathcal{C} G^\mathbb{C}_\sigma$.  Unfortunately, this is not always possible.

For the discussion of this situation we start again by considering  the universal cover $ \tilde{\pi}^\C      :\tilde{G}^\C \rightarrow G^\C$. Then $\tau$, the anti-holomorphic involution of $G^\C$ defining $G$, and $\sigma$ have natural
lifts, denoted  by $\tilde{\tau}$ and $\tilde{\sigma}$, to $ \tilde{G}^\C$.
The  fixed point group
$\tilde{K}^\C$ of $\tilde{\sigma}$ is connected and projects onto $(K^\C)^0$.
The fixed point group of $\tilde{\tau}$ in $ \tilde{G}^\C$ is generally not  connected, like in the Willmore surface  case, where  the real elements ${Fix}^{\tau}(\tilde{G}^\C) = Spin(1,n+3)$  of $\tilde{G}^\C = Spin(1,n+3, \C)$ form a non-connected group.
But it suffices to consider its connected component
$ ( {Fix}^{\tau}(\tilde{G}^\C) )^0 =  \tilde{G}$ which projects onto $G$ under $\tilde{\pi}: \tilde{G} \rightarrow G$.

From here on we will write, for convenience, $\Lambda G^0_\sigma$ for
$(\Lambda G_\sigma)^0$. Then we trivially obtain the disjoint union
\begin{equation}\label{general-Iwasawa-univ}
\Lambda \tilde{G}^{\C}_{\sigma} =
\bigcup \Lambda \tilde{G}_{\sigma}\cdot \tilde{\delta} \cdot
\Lambda^{+} \tilde{G}^{\mathbb{C}}_{\sigma},
\end{equation}
where the $\tilde{\delta}$'s simply parametrize the different double cosets.
 Note that in this equation all groups are connected.
We can (and will) assume that $\tilde{\delta} = e$ occurs.
For the corresponding double coset, since the corresponding Lie algebras add to give the full loop algebra, we obtain:

\begin{theorem}
The multiplication $\Lambda \tilde{G}_{\sigma}\times \Lambda^{+} \tilde{G}^{\mathbb{C}}_{\sigma}\rightarrow
\Lambda \tilde{G}^{\mathbb{C}}_{\sigma}$ is a real analytic map onto the connected open subset
$ \Lambda \tilde{G}_{\sigma} \cdot \Lambda^{+} \tilde{G}^{\mathbb{C}}_{\sigma}      = \mathcal{I}^{\mathcal{U}}_e \subset\Lambda \tilde{G}^{\mathbb{C}}_{\sigma}$.

\end{theorem}

From this the Iwasawa Decomposition Theorem \ref{thm-Iwasawa-0} for $(\Lambda G^\C_\sigma)^0$ follows after an application of the natural projection as above.

\begin{remark}\
\begin{enumerate}
\item We have seen above that the Iwasawa cell with middle term $I$ is open.
But also the Iwasawa cell with middle term $\delta_0 = diag(-1,1,1,1,-1,...,1)$ is open. To verify this we consider $\delta_0^{-1} \Lambda so(1, n+3)_\sigma \delta_0
\oplus \Lambda^+so(1,n+3,\C)$ and observe that the first summand is equal to $\Lambda so(1, n+3)$. As a consequence,
\[\delta_0^{-1} \Lambda SO^+(1, n+3)^0_\sigma \delta_0
\cdot  \Lambda_\mathcal{C}^+SO(1,n+3,\C)_\sigma\] is open and the claim follows.

\item  Using work of Peter Kellersch \cite{Ke1} it seems to be possible to show that there are exactly two open Iwasawa cells in this case. We will not need such a statement in this paper.
\end{enumerate}
\end{remark}

\subsection{On a complementary  solvable  subgroup of  $ SO^+(1,3) \times SO(n)$ in $ SO(1,n+3,\C)$}

We consider $K^{\mathbb{C}}$, the connected complex subgroup of $SO(1,n+3,\C)$ with Lie algebra
 $\mathfrak{so}(1,3,\C)\times \mathfrak{so}(n,\C)$ considered as
Lie algebra of matrices acting on
$\mathbb{C}^4\oplus\mathbb{C}^{n}$. (Hence we consider the ``basic" representations
 of these Lie algebras as introduced in Section 2). Clearly
 $K^{\mathbb{C}}=SO(1,3,\mathbb{C})\times SO(n,\mathbb{C}).$

\begin{theorem} There exist connected solvable subgroups $S_1 \subset SO^+(1,3,\C)$ and $S_2 \subset SO(n,\C)$ such that
 \begin{equation}
\left(SO^+(1,3) \times SO(n)\right) \times (S_1 \times S_2) \rightarrow
\left(SO^+(1,3) \cdot S_1\right) \times \left( SO(n) \cdot S_2\right)
\end{equation}
is a real analytic diffeomorphism onto an open subset of $K^\C$.
\end{theorem}
\begin{proof}
Since $SO(n)$ is a connected maximal compact subgroup of $SO(n,\mathbb{C})$,
 in $SO(n,\mathbb{C})$ we have the classical Iwasawa decomposition
 $SO(n,\mathbb{C})=SO(n)\cdot B,$  where $B$ is a solvable subgroup of $SO(n,\mathbb{C})$ satisfying
 $SO(n)\cap B=\{I\}$.

It thus suffices to consider $SO(1,3,\mathbb{C})$ and to prove the existence of a (connected solvable) subgroup $S_1$ of $SO(1,3,\C)$ such that
 \begin{equation}\label{eq-sos}
     \mathcal{S}: SO^+(1,3)\times S_1 \rightarrow SO^+(1,3)\cdot S_1
 \end{equation}
 is a real analytic diffeomorphism and
$
  SO^+(1,3)\cdot S_1
 $
  is open in $SO(1,3,\mathbb{C})$.
 Note, since the map $\mathcal{S}$ is clearly analytic and surjective, it suffices, as we will see below,  to verify that it is also open and that  $SO^+(1,3)\cap S_1=\{I\}$ holds.

 At any rate, we  need to find
a solvable Lie subalgebra $\mathfrak{s}_1$ of
$\mathfrak{so}(1,3,\mathbb{C})$
satisfying
\begin{equation} \mathfrak{so}(1,3) + \mathfrak{s}_1=\mathfrak{so}(1,3,\mathbb{C}),\ \
 \mathfrak{so}(1,3)\cap\mathfrak{s}_1=\{0\}.
\end{equation}

Set
\[\mathfrak{s}_1=\left\{\left.\left(
    \begin{array}{cccc}
      0 & i a_{12} & a_{13} & ia_{13} \\
      ia_{12} & 0 & a_{23} & ia_{23} \\
      a_{13} & -a_{23}  & 0 & ia_{34} \\
      i a_{13} & -ia_{23} & -ia_{34} & 0 \\
    \end{array}
  \right)\right|\ a_{12},a_{34}\in\R, a_{13}, a_{23}\in\C\right\}.
\]
We see that $\mathfrak{so}(1,3) \cap\mathfrak{s}_1=\{0\}$ and $\mathfrak{so}(1,3)^{\C}=\mathfrak{so}(1,3)\oplus\mathfrak{s}_1$ hold.
 It is straightforward to see  that $\mathfrak{s}_1$ is a solvable Lie algebra.
Let $S_1$ be the connected Lie subgroup of $SO(1,3,\C)$ with Lie algebra $Lie(S_1)=\mathfrak{s}_1$. So we have that the map $ \mathcal{S}$ is  a local  diffeomorphism near the identity element  by  Chapter II, Lemma 2.4 of \cite{Helgason}. This also implies that the map $\mathcal{S}$ is open.

 Next we finally show that  $SO^+(1,3)\cap S_1=I$ holds.
 We recall  that the exponential map $\exp:\mathfrak{so}(1,3)\rightarrow SO^+(1,3)$ is surjective.  Then every element of $SO^+(1,3)\cdot S_1$ has the form
$\exp( \mathfrak{A}) \exp( \mathfrak{B})\exp(\mathfrak{C}),$ with $\mathfrak{A}\in\mathfrak{so}(1,3)$,
  $\mathfrak{B}$ contained in the abelian subalgebra of  the $2\times2-$ diagonal blocks in $\mathfrak{s}_1 $, and
 $\mathfrak{C}$  in the nilpotent subalgebra of $\mathfrak{s}_1 $  consisting of the  ``off-diagonal''  blocks
(Note that for every  off-diagonal block $Q$ in $\mathfrak{s}_1$ we have $Q^2 = 0$).   Let  $\exp( \mathfrak{A}) \exp( \mathfrak{B})\exp(\mathfrak{C})\in SO^+(1,3)\cap S_1$.  Then $ \exp( \mathfrak{B})\exp(\mathfrak{C})= \exp(\overline{ \mathfrak{B}})\exp(\overline{\mathfrak{C}})$ and
  \[\exp(\overline{ \mathfrak{B}})^{-1}\exp( \mathfrak{B})= \exp(\overline{\mathfrak{C}})\exp(\mathfrak{C})^{-1}\]
follows.
  As a consequence, $\exp(\overline{ \mathfrak{B}})^{-1}\exp( \mathfrak{B})= \exp(\overline{\mathfrak{C}})\exp(\mathfrak{C})^{-1}=I_4$, i.e.,
  $\exp( \mathfrak{B})=\exp(\overline{ \mathfrak{B}})$ and $\exp(\mathfrak{C})=\exp(\overline{\mathfrak{C}})$.  The definition of    $ \mathfrak{s}_1 $ now implies  $\exp(\mathfrak{B})=\exp(\mathfrak{C})=I_4$.

 To see that the inverse map is real analytic we take a small neighbourhood in  $SO^+(1,3) \cdot  S_1$ of the form $gVs$, where $V$ is a small neighbourhood of the identity $I$. Since locally near $I$ our map is a real analytic diffeomorphism, the claim follows.
\end{proof}

\begin{remark}
We point out that the set $SO^+(1,3)S_1$ is not all of
$SO(1,3,\C)$. For example
\[\left(
               \begin{array}{cccc}
                 \frac{\sqrt{2}}{2} & 0 & \frac{i\sqrt{2}}{2} & 0 \\
               0  & 1 & 0 & 0 \\
                 \frac{i\sqrt{2}}{2} & 0 &\frac{\sqrt{2}}{2} & 0 \\
                0 & 0 & 0 & 1 \\
               \end{array}
             \right)
\]
 is an element of $SO^+(1,3,\C)$ which is not contained in $ SO^+(1,3)S_1$.
\end{remark}

{\small{\bf Acknowledgements}\ \ The second named author is partly supported by the Project 11571255 of NSFC.
 The second named author is thankful to the ERASMUS MUNDUS TANDEM Project for the financial supports to visit the TU M\"{u}nchen.
}

{\footnotesize

\def\refname{Reference}

}
\vspace{2mm}

{\footnotesize \begin{multicols}{2}   
Josef F. Dorfmeister

Fakult\" at f\" ur Mathematik,

TU-M\" unchen, Boltzmann str. 3,

D-85747, Garching, Germany

{\em E-mail address}: dorfm@ma.tum.de\\

Peng Wang

College of Mathematics \& Informatics, FJKLMAA,

Fujian Normal University, Qishan Campus,

Fuzhou 350117, P. R. China

{\em E-mail address}: {pengwang@fjnu.edu.cn}

\end{multicols}}
\end{document}